\documentclass[10pt,reqno,a4paper]{amsart}

\let\oldtocsection=\tocsection

\let\oldtocsubsection=\tocsubsection

\let\oldtocsubsubsection=\tocsubsubsection

\renewcommand{\tocsection}[2]{\hspace{0em}\oldtocsection{#1}{#2}}
\renewcommand{\tocsubsection}[2]{\hspace{2em}\oldtocsubsection{#1}{#2}}
\renewcommand{\tocsubsubsection}[2]{\hspace{2em}\oldtocsubsubsection{#1}{#2}}

\usepackage{amsmath,amsthm,amsfonts,amssymb}

\usepackage{enumitem}

\usepackage{marginnote}

\usepackage{fancyhdr}

\usepackage{graphicx}

\usepackage{pdfpages}

\usepackage{caption}

\usepackage{tikz}
\usetikzlibrary{matrix}
\usetikzlibrary{cd}

\usepackage{xcolor}

\usepackage{rotating}

\numberwithin{figure}{section}
\numberwithin{equation}{section}

\newtheorem{thm}{Theorem}[section]

\newtheorem{defn}[thm]{Definition}

\newtheorem{lmm}[thm]{Lemma}

\newtheorem{remark}[thm]{Remark}

\newcommand{\sgn}{\operatorname{sign}}

\newcommand{\Neg}{\operatorname{Neg}}
\newcommand{\interior}{\operatorname{Int}}
\newcommand{\mult}{\operatorname{mult}}

\renewcommand{\mod}{\operatorname{mod\,}}

\newcommand{\absv}[1]{\lvert #1 \rvert}
\newcommand{\abs}[1]{\lvert\lvert #1 \rvert\rvert}

\newcounter{reminder}

\graphicspath{ {./images/}}

\title{Irrational rotation dynamics for unimodal maps}
\author{Konstantin Bogdanov}
\author{Alexander Bufetov}
\address[Konstantin Bogdanov]{Institute of Mathematics of Polish Academy of Sciences, ul. Śniadeckich 8, 00-656 Warsaw, Poland}
\email{kbogdanov@impan.pl}
\address[Alexander Bufetov]{CNRS, Aix-Marseille Universit\'e, Centrale Marseille, Institut de Math\'ematiques de Marseille, UMR7373, 39 Rue F. Joliot Curie 13453, Marseille, France; Steklov Mathematical Institute of RAS, Moscow, Russia; Institute for Information Transmission Problems, Moscow, Russia}
\email{alexander.bufetov@univ-amu.fr; bufetov@mi-ras.ru}

\begin{document}
	
\begin{abstract}
	The first result of the paper (Theorem 1.1) is an explicit construction of  unimodal maps that are semiconjugate, on the post-critical set,  to the  circle rotation by an arbitrary irrational angle $\theta\in(3/5,2/3)$. Our construction is a generalization of  the construction by Milnor and Lyubich \cite{LM} of the Fibonacci unimodal maps semi-conjugate to the circle rotation by the golden ratio. Generalizing a theorem by Milnor and Lyubich for the Fibonacci map, we  prove that the Hausdorff dimension of the post-critical set of our unimodal maps is $0$, provided the denominators of the continued fraction of $\theta$ are bounded (Theorem 1.2) or, in the case of quadratic polynomials, have sufficiently slow growth (Theorem 1.3).   
\end{abstract}

\maketitle
	
\tableofcontents

\addtocontents{toc}{\protect\setcounter{tocdepth}{1}}

\section{Introduction}
Along with circle homeomorphisms, real \emph{unimodal maps} induce simplest non-trivial 1-dimensional dynamical systems. For a closed interval $I\subset\mathbb{R}$, a continuous map $f:I\to I$ is called unimodal with the extremum point $x_0$ of $f$ is strictly increasing on one side of $x_0$ and strictly decreasing on the other. One of natural and combinatorially full families of such maps are quadratic polynomials $f(x)=x^2+c$ with real parameters $c$ belonging to the Mandelbrot set (i.e., the critical orbit $\mathcal{O}:=\{f^n(0)\}_{n=0}^\infty$ is bounded).

Dynamical properties of the unimodal maps depend heavily on their class (e.g.\ $C^1$ or $C^2$) and the behaviour of the critical orbit. In \cite{LM} Lyubich and Milnor described Fibonacci real unimodal maps. As a defining property serves the combinatorial restriction on $\mathcal{O}$: the times of the closest recurrence of $f^n(x_0)$ to $x_0$ are exactly the Fibonacci numbers $1,2,3,5,8,13,...$ In absence of wandering intervals this property defines uniquely the topology on $\overline{\mathcal{O}}$ (including the order of points of $\mathcal{O}$ in $\mathbb{R}$). More precisely, $\overline{\mathcal{O}}$ is an explicitly given Cantor set and $f|_{\overline{\mathcal{O}}}$ is semiconjugate to the circle rotation by the golden ratio $\frac{1+\sqrt{5}}{2}$. Note that
$$\frac{1+\sqrt{5}}{2}\equiv\frac{1+\sqrt{5}}{2}-1=[1,1,1,...]:=\cfrac{1}{1+\cfrac{1}{1+\cdots}},$$
and the denominators $q_n$ of the truncated fractions $[\underbrace{1,1,...,1}_\text{$n$ times}]$ are the Fibonacci numbers.

Hence, a natural question arises: could the theory be generalized by replacement of $[1,1,1,...]$ by an arbitrary continued fraction $[a_1,a_2,a_3,...]$ and what can be said about unimodal maps with times of closest recurrence coinciding with denominators of $\theta=[a_1,a_2,...,a_n]$? It turns out that this condition is not sufficient to determine the order of points in $\mathcal{O}$. This can roughly be explained as follows: if the sequence $\{a_n\}$ contains many terms bigger than $1$, then the sequence $\{q_n\}$ (of denominators of the truncated continued fraction) grows too fast and the order of points $\{x_0,x_1,...,x_{q_k}\}$ on the real line does not have a big impact on the position of points $x_n$ with $n$ bigger than \ $q_k+q_{k-1}$. However, there is a natural recursive side-condition (see a somewhat elaborate Definition~\ref{defn:recurrence}) under which the topology on $\overline{\mathcal{O}}$ is determined in a canonical way (depending only on the irrational angle $\theta=[1,1,1,a_4,a_5,...], a_n\in\mathbb{N},n\geq 4$). Based on the initial choice of $\theta$, we call such maps \emph{$\theta$-recurrent}. In particular, the Fibonacci maps are $(\sqrt{5}-1)/2$-recurrent. Moreover, we have the following

\begin{thm}[$\theta$-recurrent maps]
	\label{thm:theta_recurrence}
	Let $\theta\in(3/5,2/3)$ be irrational. 	
	\begin{enumerate}		
		\item  There exists one and only one real quadratic polynomial $x^2+c$ which is $\theta$-recurrent. Moreover, we have
		$$\absv{x_1}>\dots>\absv{x_{q_{n-1}}}>\absv{x_{a_{n+1}q_n}}>\absv{x_{(a_{n+1}-1)q_n}}>\dots>\absv{x_{q_n}}>\absv{x_{a_{n+2}q_{n+1}}}>\dots>0.$$
		
		\item If $f:I\to I$ is a $\theta$-recurrent and has no homtervals, then the closure $\overline{\mathcal{O}}$ of the critical orbit is a Cantor set and the restriction $f|_{\overline{\mathcal{O}}}$ is semiconjugate to the circle rotation by the angle $\theta$.
	\end{enumerate}
\end{thm}

As for the Fibonacci map, under certain smoothness condition for $\theta$-recurrent maps with $\theta$ of bounded type (i.e.\ with bounded denominators of its continued fraction) it is possible to estimate the asymptotics of its points of closest recurrence and use it to compute the Hausdorff dimension of $\overline{\mathcal{O}}$.

Denote by $\delta_n^i,0<i\leq a_{n+1}$ the ratio of $\absv{x_{iq_n}}$ and its closest left neighbour in the inequalities of Theorem~\ref{thm:theta_recurrence}. So, $\delta_n^i<1$.

\begin{thm}[Hausdorff dimension and asymptotics]
	\label{thm:hausdorff_dim_bounded}
	Let $\theta\in(3/5,2/3)$  be irrational and of bounded type, and $f:I\to I$ be $C^2$ smooth with non-flat critical point. The Hausdorff dimension of $\overline{\mathcal{O}}$ is equal to $0$.
	
	If, additionally, $f$ is equal to $x^2+c$ near the origin, the following asymptotic formulas hold as $n\to\infty$.
	$$\left(\delta_{n+2}^{a_{n+3}}\right)^{2^{a_{n+2}}}\sim\left(\delta_{n+1}^{a_{n+2}}\right)^{2^{a_{n+1}}-1}\delta_n^{a_{n+1}},$$
	and for $0<i<a_{n+1}$,
	$$\delta_n^i\sim\left(\delta_{n+1}^{a_{n+2}}\right)^{2^i}.$$
	The sign ``$\sim$'' is understood as equality modulo factor $1+O(p^n)$ for some $0<p<1$.
\end{thm}

In case of a quadratic polynomial one can allow even a bit more. For an angle $\theta=[a_1,a_2,...]$, let $\{N_k\}_{k=1}^{\infty}=\{N_k(\theta)\}_{k=1}^{\infty}$ be a strictly increasing sequence of integers such that each $a_{N_k}$ is strictly bigger than all $a_i$ with smaller indices $i<N_k$.  
\begin{thm}[Hausdorff dimension for unbounded type]
	\label{thm:hausdorff_dim_unbounded}
	Let $\theta=[a_1,a_2,...]\in(3/5,2/3)$  be irrational of unbounded type such that $N_{k+1}-N_k> 2^{(5+\tau)a_{N_{k+1}}}$ for all $k$ big enough and some $\tau>0$, and $f(x)=x^2+c$ be $\theta$-recurrent. The Hausdorff dimension of $\overline{\mathcal{O}}$ is equal to $0$.
\end{thm}

\begin{remark}
	It would be interesting if the estimate on the growth of $N_k$ in Theorem~\ref{thm:hausdorff_dim_unbounded} could be replaced by the optimal one.
\end{remark}

The proof of Theorem~\ref{thm:hausdorff_dim_bounded} generalizes the analogous proof in \cite{LM} which on its own uses ideas of Sullivan \cite{S}. First, in Section~\ref{sec:a_priori} we provide \emph{a priori bounds} for a fixed $\theta$-recurrent map $f$. Next, in Section~\ref{sec:Hausdorff} we show that if $\delta_n^i$ appear arbitrarily small, at some point they start to decrease with the exponential speed. From this follow very precise bounds on certain iterates of $f$, as well as asymptotics of Theorem~\ref{thm:hausdorff_dim_bounded}. By an additional computation we get bounds on the Hausdorff $\varepsilon$-measure for each $\varepsilon>0$ and this proves the theorem under condition that the geometry of $f$ degenerates. Next step is to prove that this is always the case.

In Section~\ref{sec:renorm} we introduce a renormalization procedure for a special class of functions to which $f$ belongs after a surgery not changing $\delta_n^i$'s. The map $f$ (after surgery) is infinitely renormalizable. Unlike for the Fibonacci maps in \cite{LM}, where the renormalization is basically defined using the fact that the infinite Fibonacci word is a fixed point of a substitution, this does not work for general $\theta$. Hence, our definition if modeled on the renormalization of circle rotations (see e.g.\ \cite{Arnoux}).

If the geometry of $f$ does not degenerate, then the sequence of its renormalizations has a limit point $g$ which is $\theta'$-recurrent, with $\theta'$ of bounded type. This limit point after an additional renormalization can be made a polynomial-like map of type (2,1) (see \cite{LM}). All such maps (with the same $\theta'$) are quasimetrically conjugate and one can construct explicitly an example of such a map with degenerating geometry which implies that $\delta_n^i$'s of $f$ are arbitrarily small. The proofs we are referring to in this paragraph are the same or the same after elementary correction as for the Fibonacci maps, so for them we only provide a reference.

For $f(x)=x^2+c$ we do not need to consider a limit of renormalization. In fact, one can do an explicit one-time renormalization replacing $f$ by a polynomial-like map of type (2,1). By changing within its quasi-symmetrical conjugacy class we can assume the first finitely many $\delta_n^i$'s arbitrarily small. So, proof of Theorem~\ref{thm:hausdorff_dim_unbounded} is based on more accurate computations of when the exponential decrease of $\delta_n^i$ begins. Roughly speaking: run this exponential decrease of $\delta_n^i$ by making $a_n$ bigger at the cost of making the multiplicative factor bigger as well.

\subsection*{Acknowledgements}
We are deeply grateful to Vladlen Timorin for inspiring and fruitful discussions of the project. Both authors were supported by the ANR grant ANR-18-CE40-0035 REPKA.

\section{Dynamics of irrational rotations}
\label{sec:theta_recurrent}

To begin with, we describe what we understand by the dynamics of irrational rotation for a unimodal map.

Let $I=[-1,1]$ and $f:I\to I$ be a unimodal map with the minimum point at $0$ and $f(-1)=f(1)=1$. Further, we denote by $\mathcal{O}=\{x_n\}_{n=0}^\infty$ the critical orbit of $f$, that is, $x_n=f^n(0)$.

Next, for a nonzero $x\in I$, we use the notation $x'$ for the other point satisfying $f(x')=f(x)$, while for $x=0$, $x':=0$. For a pair $y,z\in I$ we write $\abs{y}<\abs{z}$ if $f(y)<f(z)$. Denote also by $I_x$ the closed non-oriented interval $[x,x']$ (possibly consisting of one point $0$). For $y,z\notin I_x$, we say that $y$ is closer to $I_x$ than $z$ if $\abs{y}<\abs{z}$.

In \cite{LM} the Fibonacci real quadratic polynomial was defined as a polynomial for which the closest recurrence of its critical orbit happens $\{x_n\}_{n=0}^\infty$ for those $n$ that are Fibonacci numbers 1,2,3,5,8,... Since the golden ratio $\varphi=\frac{\sqrt{5}+1}{2}$ satisfies $\varphi=1+[1,1,1,...]$, for an irrational angle $\theta\in(0,1)$ it seems natural to represent it via its continued fraction $$\theta=[a_1,a_2,a_3,...]=\cfrac{1}{a_1+\cfrac{1}{a_2+\cfrac{1}{a_3+\cdots}}}, a_n\in\mathbb{N}$$ and to define a $\theta$-recurrent unimodal map $f$ as the one having closest recurrence of the critical point at times $q_n$ where $q_n$ is the denominator of $\theta_n:=[a_1,a_2,...,a_n]$ (recall that $q_n=a_n q_{n-1}+q_{n-2}$). Unfortunately, in such setting there are, generally speaking, many different $\theta$-recurrent maps not necessarily having critical orbit with a self-similar structure as in case of the Fibonacci map.

Therefore we need to use a somewhat more specific definition of closest recurrence. We start with a ``finite'' version.

\begin{defn}[$\theta$-recurrence for intervals]
	\label{defn:recurrence_intervals}
	Let $f:I\to I$ be a unimodal map, $x\in I$, $\theta=[a_1,a_2,...,a_N]$ be a finite continued fraction, and $q_k$ be the denominator of $[a_1,a_2,...,a_k]$. We say that $x$ has \emph{$\theta$-recurrence} if the following conditions hold:
	\begin{enumerate}
		\item for $k=1,2,3,...,q_N$, points $x_k=f^k(x)$ are not in $I_x$,
		\item times of closest recurrence of $x$ to $I_x$ are exactly $q_1,q_2,...,q_N$,
		\item for $m=0,1,...,a_N-1$, a point $x_{m q_{N-1}}$ has $[a_1,a_2,...,a_{N-1}]$-recurrence.  
	\end{enumerate}
\end{defn}

Definition~\ref{defn:recurrence_intervals} is fashioned to make use of the formula $q_n=a_n q_{n-1}+q_{n-2}$ to prescribe which parts of the orbit $\{x_k\}$ have to be ``similar'': the parts from $0$ to $q_{N-1}$, from $q_{N-1}$ to $2q_{N-1}$,\ ...\ , from $(a_N-1)q_{N-1}$ to $a_N q_{N-1}$. So the orbit of $x$ can be split into $a_N+1$ consecutive blocks, first $a_N$ of which represent a $[a_1,a_2,...,a_{N-1}]$-recurrent orbit.

\begin{defn}[$\theta$-recurrence]
	\label{defn:recurrence}
	Let $f:I\to I$ be a unimodal map and $\theta=[a_1,a_2,a_3,...]$ be an infinite continued fraction. We say that the critical value $x_0$ has \emph{$\theta$-recurrence} (or $f$ is $\theta$-recurrent) if $x_0$ is $\theta_n$-recurrent for every $\theta_n=[a_1,a_2,...,a_n]$
\end{defn}

Clearly, for $\theta=[1,1,1,...]$ Definition~\ref{defn:recurrence} coincides with the definition via closest recurrence at times that are Fibonacci numbers.

Note that for any $\theta$ and a $\theta$-recurrent map it is always true that $x_1<0,x_2>0$ and $x_3\in (x_1,x_2)$ --- otherwise we get a non-recurrent dynamics of the critical point. 

Now, we want to exclude those irrational angles $\theta=[a_1,a_2,a_3,...]$ for which there is no $\theta$-recurrent map. Trivially, the first time of closest recurrence has to be equal to $1$. If $x_2>x_1'$, then $x_2<x_3<x_4<...$, so the second time of closest recurrence is $2$. This is possible only in cases $a_1=2$ or $a_1=1,a_2=1$. However, these cases correspond to the conjugate dynamics (rotation by $\theta$ and $-\theta$), so we may agree to deal with only one of them. From now on, we assume $a_1=1,a_2=1$. Also, $a_3=1$: indeed, we have $q_1=1,q_2=2,q_3=2a_3+1\geq 5$, and $x_1'>x_2$. If $a_3>1$, then $3$ is not a time of closest recurrence, so $x_3\in [x_1,x_2']$. Further, $x_4$ has to be in $f([x_1,x_2'])=[x_3,x_2]$, but since $4$ is also not a time of closest recurrence, we have $x_4\in[x_3,x_2']$. So, the interval $[x_3,x_2']$ is mapped by $f$ into itself. Hence, such map cannot be $\theta$-recurrent.

\begin{defn}[Admissible angles]
	\label{defn:admissible_angles}
	An irrational angle $\theta=[a_1,a_2,a_3,...]$ is called \emph{admissible} if $a_1=a_2=a_3=1$.
\end{defn}

The goal of this section is to show that every admissible angle $\theta$ can be realized by a $\theta$-recurrent unimodal map and each $\theta$ uniquely defines the order of $\{x_k\}$ in $I$. To describe this order we need to consider a \emph{number system associated to $\theta$}, which is completely analogous to the Fibonacci number system.

Given an irrational $\theta=[a_1,a_2,...]$ with $a_1=1$, every integer $k$ can be written in a unique way as a sum $k=\sum_{n=1}^{\infty}\gamma_n q_n$ where $0\leq\gamma_n\leq a_{n+1}$, only finitely many $\gamma_n$'s are nonzero and if the digit $\gamma_n$ is equal to $a_{n+1}$, then $\gamma_{n-1}=0$. One represents such $k$ as $[\gamma_1\gamma_2\gamma_3\cdots]$ starting from the smallest term. The representation $[00...00\gamma_n 00\cdots]$ with $\gamma_n=1$ corresponds to $k=q_n$. Note that one can also consider formal infinite sums $\varkappa=\sum_{n=1}^{\infty}\gamma_n q_n$ where $0\leq\gamma_n\leq a_{n+1}$ and if the digit $\gamma_n$ is equal to $a_{n+1}$, then $\gamma_{n-1}=0$. Such infinite sums are limits of finite words $[\gamma_1\gamma_2...\gamma_m00\cdots]$ in the product topology.

\begin{thm}[Signs of $\{x_k\}$]
	\label{thm:signs}
	For the critical orbit $\{x_k\}_{k>0}$ of a $\theta$-recurrent map $f$ the following relations are true:
	\begin{enumerate}
		\item $x_{q_n}<0$ if $n\equiv 0,1 \mod 4$ and $x_{q_n}>0$ otherwise;
		\item if $k=\gamma_n q_n$ with $\gamma_n>1$, then $x_{q_n}$ and $x_{\gamma_n q_n}$ have opposite signs;
		\item if $k=\sum_{n=m}^{\infty}\gamma_n q_n$ with $\gamma_m>0$, then $x_k$ and $x_{\gamma_m q_m}$ have the same sign.
	\end{enumerate}  
\end{thm}

An example of the critical orbit of $\theta$-recurrent map for $\theta=[1,1,1,3,2,...]$ is provided on Picture~\ref{pic:example}.

We will also need a rather simple
\begin{lmm}
	\label{lmm:number_of_negative_terms}
	For a finite or infinite sequence $\{x_k\}_{k>0}$, let $\Neg(m)$ denote the number of $0<k<m$ such that $x_{k}<0$. If $\{x_k\}_{k>0}$ satisfies $(1),(2),(3)$ of Theorem~\ref{thm:signs} for $0<k<q_n$, then $\Neg(q_n)$ is even if $n\equiv 0,1 \mod 4$ and odd otherwise.
\end{lmm}
\begin{proof}
	Lemma is true for $n=1,2$. Then for $n>2$ holds
	$$\Neg(q_n)=\Neg(a_nq_{n-1}+q_{n-2})=a_n\Neg(q_{n-1})+\delta_n+\Neg(q_{n-2}),$$
	where
	\begin{equation}
	\delta_n =
		\begin{cases}
			1 & \text{if $x_{q_{n-1}}<0$}\\
			a_n-1 & \text{if $x_{q_{n-1}}>0$}
		\end{cases}       
	\end{equation}
	Hence, by induction on $n$ we have
	\begin{equation}
		\Neg(q_n) \equiv
		\begin{cases}
			a_n\cdot 1+a_n-1+1\equiv 0\mod 2 & \text{if $n\equiv 0\mod 4$}\\
			a_n\cdot 0+1+1\equiv 0\mod 2 & \text{if $n\equiv 1\mod 4$}\\
			a_n\cdot 0+1+0\equiv 1\mod 2 & \text{if $n\equiv 2\mod 4$}\\
			a_n\cdot 1+a_n-1+0\equiv 1\mod 2 & \text{if $n\equiv 3\mod 4$}
		\end{cases}   
	\end{equation}

	To prove part $(b)$ note that $q_n-1=a_nq_{n-1}+a_{n-2}q_{n-3}+a_{n-4}q_{n-5}+...$ and its smallest term is either $q_1=1$ or $q_2=2$ depending on whether $n$ is even or odd, respectively.
\end{proof}

Since Theorem~\ref{thm:signs} determines the signs of the whole orbit $\{x_k\}$, it automatically determines the order $\{x_k\}$ in $I$ (unless the sequence of signs is eventually periodic, but we will see later that this does not happen). The proof uses a similar inductive step as in \cite{LM}. We formulate it in form of Lemma~\ref{lmm:signs}. Theorem~\ref{thm:signs} follows trivially.

\begin{lmm}[Sign of $x_k$ for finite $\theta$]
	\label{lmm:signs}
	Let $\theta=[a_1,a_2,a_3,...,a_n], n\geq 4$ be a finite continued fraction and $x_0\in I$ (not necessarily equal to zero) has $\theta$-recurrence. Statements $(1),(2),(3)$ in Theorem~\ref{thm:signs} are true for $x_k$ with $k\in \{1,2,\dots,q_n-1\}\setminus\{q_{n-1},2q_{n-1},...,a_nq_{n-1}\}$.
\end{lmm}

\begin{proof}
	
	We use the induction by $n$. Let us first do the induction step for $n\geq 4$, then describe the basis of induction.
	
	Thus, let the statement of the lemma be true for all continued fractions $\theta$ of length less or equal than $n$, and prove it for $n+1$. Since $x_0$ has $[a_1,a_2,a_3,...,a_{n+1}]$-recurrence, by Definition~\ref{defn:recurrence_intervals} each of the points $x_{iq_n}, i\in \{0,1,...,a_{n+1}-1\}$ has $[a_1,a_2,a_3,...,a_n]$-recurrence. In the block of length $q_n$ starting from $iq_n$, the induction hypothesis determines signs of all $x_k$ except for $k=iq_n,iq_n+q_{n-1},iq_n+2q_{n-1},...,iq_n+a_nq_{n-1}$, and these signs satisfy $(1),(2),(3)$ (for step $n+1$) because each of them in the number system associated with $[a_1,a_2,a_3,...,a_{n+1}]$ (defined for integers less or equal than $q_{n+1}$) has dominating (smallest) term less than $q_{n-1}$.
	
	\begin{figure}[h]
		\includegraphics[width=\textwidth]{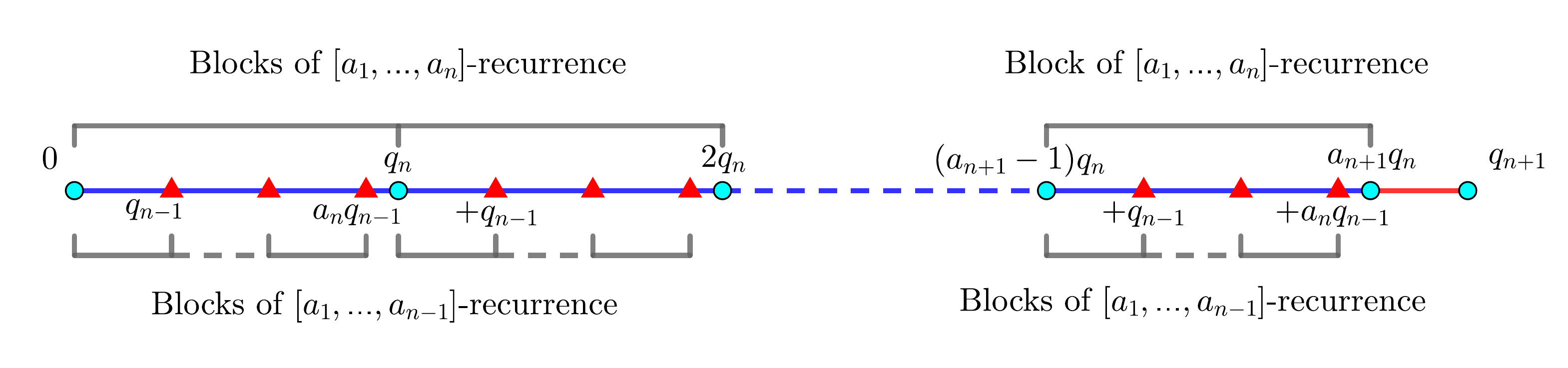}
		\caption{Induction step}
		\label{pic:induction}
	\end{figure} 
	
	Hence, one has to determine signs of points $x_{iq_n+kq_{n-1}}$ and $x_{a_{n+1}q_n+l}$ where $0\leq i<a_{n+1},0<k\leq a_n,0<l<q_{n-1}$. It is much easier when visualized: on Picture~\ref{pic:induction} blue lines correspond to the indices for which $n$-hypothesis determines their signs. The red line and red triangles correspond to indices for which we need to determine their sign in order to do the induction step. Finally, in light blue are marked those points for which the sign is not know and does not need to be determined on this induction step.
	
	\textbf{Induction step for $\mathbf{a_n>1}$.} We have the following inequalities
	$$\abs{x_{iq_n}}<\abs{x_{a_{n+1}q_n}}<\abs{x_{(a_{n+1}-1)q_n+kq_{n-1}}}$$
	for every $0\leq i<a_{n+1},0<k<a_n$. These inequalities are immediate consequences of the definition of $\theta$-recurrence. Terms in the first one are the starting or the ending point of blocks corresponding to $[a_1,...,a_n]$-recurrence --- and the starting point in each such block separating the two indices is closer to $0$ the ending one. In the second inequality $x_{a_{n+1}q_n}$ is again considered as the ending point of a block of $[a_1,...,a_n]$-recurrence, whence every intermediate point in this block is further from $0$, in particular, this is true for points $x_{(a_{n+1}-1)q_n+kq_{n-1}}$.
	
	By the induction hypothesis for $1\leq m< q_{n-1}$ the $f^m$-images of all points involved in the inequalities, except $x_{a_{n+1}q_n}$, have the same signs, that is, they belong to the same domain of monotonicity of $f$. We have
	$$x_{a_{n+1}q_n+m}\in[x_{iq_n+m},x_{(a_{n+1}-1)q_n+kq_{n-1}+m}],$$
	that is, $x_{a_{n+1}q_n+m}$ has the same sign as $x_m$ and satisfies the hypothesis for step $n+1$.
	
	Next, taking $m=q_{n-1}$, we obtain
	$$x_{q_{n+1}}=x_{a_{n+1}q_n+q_{n-1}}\in[x_{iq_n+q_{n-1}},x_{(a_{n+1}-1)q_n+(k+1)q_{n-1}}].$$
	But since $q_{n+1}$ is a time of closest recurrence, the endpoints of each interval must have opposite signs. Thus, for $0\leq i<a_{n+1},0<k<a_n$, points $x_{iq_n+q_{n-1}}$ have the same sign while points $x_{(a_{n+1}-1)q_n+(k+1)q_{n-1}}$ have the same opposite sign.
	
	If $a_{n+1}>1$, consider a pair of points $x_{iq_n},x_{iq_n+kq_{n-1}}$ for some $0\leq i<a_{n+1}-1,0<k<a_n$. Note that
	\begin{enumerate}
		\item $\abs{x_{iq_n}}<\abs{x_{iq_n+kq_{n-1}}}$,
		\item $\abs{x_{(a_{n+1}-1)q_n}}<\abs{x_{(a_{n+1}-1)q_n+kq_{n-1}}}$,
		\item for $1\leq m<q_{n-1}$ the $f^m$-images of all points involved in these inequalities have the same signs,
		\item $x_{iq_n+q_{n-1}}$ and $x_{(a_{n+1}-1)q_n+q_{n-1}}$ have the same sign.
	\end{enumerate}
	Hence, since $x_{(a_{n+1}-1)q_n+q_{n-1}}$ and $x_{(a_{n+1}-1)q_n+(k+1)q_{n-1}}$ have opposite signs, $0$ and $x_{iq_n+(k+1)q_{n-1}}$ are on the same side of $x_{iq_n+q_{n-1}}$. But by $\theta$-recurrence we get
	$\abs{x_{iq_n+(k+1)q_{n-1}}}>\abs{x_{iq_n+q_{n-1}}}$, so $x_{iq_n+q_{n-1}}$ and $x_{iq_n+(k+1)q_{n-1}}$ must have opposite signs. Thus, we have shown that for $0\leq i<a_{n+1},0<k<a_n$, points $x_{iq_n+q_{n-1}}$ have the same sign while points $x_{iq_n+(k+1)q_{n-1}}$ have the same opposite sign. We have to determine this sign.
	
	Consider a pair of points $x_0,x_{q_{n-1}}$. We have $\abs{x_0}<\abs{x_{q_{n-1}}}$ and for $1\leq m<q_{n-1}$ the $f^m$-images of both points have the same signs, whence $f^{q_{n-1}-1}$ is orientation-preserving on $[x_1,x_{q_{n-1}+1}]$ if $\Neg(q_{n-1})$ is even and orientation-reversing otherwise. But since $x_{q_{n-1}}$ and $x_{2q_{n-1}}$ have opposite signs, $x_{q_{n-1}}$ must be negative if $f^{q_{n-1}-1}$ is orientation-preserving and positive otherwise. So from Lemma~\ref{lmm:number_of_negative_terms} $x_{q_{n-1}}<0$ if $n-1\equiv 0,1\mod 4$ and $x_{q_{n-1}}>0$ otherwise.
	
	\textbf{Induction step for $\mathbf{a_n=1}$.} In this case we need to determine signs of $x_{iq_n+q_{n-1}}$ and $x_{a_{n+1}q_n+l}$ where $0\leq i<a_{n+1},0<l<q_{n-1}$.
	
	First, we deal with points $x_{a_{n+1}q_n+l}$. By $\theta$-recurrence we have
	$$\abs{x_{(a_{n+1}-1)q_n}}<\abs{x_{a_{n+1}q_n}}<\abs{x_{(a_{n+1}-1)q_n+q_{n-1}}}.$$
	For $1\leq m<q_{n-2}$ the $f^m$-images of $x_{(a_{n+1}-1)q_n}$ and $x_{(a_{n+1}-1)q_n+q_{n-1}}$ have the same signs, whence
	$$x_{a_{n+1}q_n+m}\in[x_{(a_{n+1}-1)q_n+m},x_{(a_{n+1}-1)q_n+q_{n-1}+m}].$$
	Thus, the points $x_{a_{n+1}q_n+m}$ have the same sign as $x_m$ which agrees with the hypothesis.
	Taking $m=q_{n-2}$, we obtain
	$$x_{a_{n+1}q_n+q_{n-2}}\in[x_{(a_{n+1}-1)q_n+q_{n-2}},x_{a_{n+1}q_n}].$$ Denote by $y_1$ one of the points $x_{(a_{n+1}-1)q_n+q_{n-2}},x_{a_{n+1}q_n}$ which has the same sign as $x_{a_{n+1}q_n+q_{n-2}}$ (pick bigger of them if signs coincide). This means that 
	$$\abs{x_0}<\abs{x_{a_{n+1}q_n+q_{n-2}}}<\abs{y_1}.$$
	
	If $a_{n-1}>1$ we proceed: for $1\leq m<q_{n-2}$ the $f^m$-images of $x_0$ and $y_1$ have the same signs, whence
	$$x_{a_{n+1}q_n+q_{n-2}+m}\in[x_m,f^m(y_1)].$$
	The points $x_{a_{n+1}q_n+q_{n-2}+m}$ have the same sign as $x_m$ which agrees with the hypothesis.
	Taking $m=q_{n-2}$, we obtain
	$$x_{a_{n+1}q_n+2q_{n-2}}\in[x_{q_{n-2}},f^{q_{n-2}}(y_1)].$$ Denote by $y_2$ one of the points $x_{q_{n-2}},f^{q_{n-2}}(y_1)$ which has the same sign as $x_{a_{n+1}q_n+q_{n-2}}$ (pick bigger of them if signs coincide). This means that 
	$$\abs{x_0}<\abs{x_{a_{n+1}q_n+2q_{n-2}}}<\abs{y_2}.$$
	
	After repeating this procedure $a_{n-1}-1$ times, we see that the signs of points $x_{a_{n+1}q_n+kq_{n-2}}+m$ with $0<k<a_{n-1}-1, 0<m<q_{n-2}$ satisfy the hypothesis and have a system of inequalities
	$$\abs{x_0}<\abs{x_{a_{n+1}q_n+kq_{n-2}}}<\abs{y_k}$$
	where $1<k\leq a_{n-1}$.
	
	Note that $y_{a_{n-1}}$ is equal either to $x_{(a_{n+1}-1)q_n+a_{n-1}q_{n-2}}$, or to $x_{a_{n+1}q_n+(a_{n-1}-1)q_{n-2}}$, or to $x_{jq_{n-2}}$ with $0<j<a_{n-1}$. But the latter two cases are simply not possible because they imply that $x_{q_{n+1}}=x_{a_{n+1}q_n+a_{n-1}q_{n-2}}+q_{n-3}$ belongs either to $[x_m,x_{a_{n+1}q_n+(a_{n-1}-1)q_{n-2}+q_{n-3}}]$ or to $[x_m,x_{jq_{n-2}+q_{n-3}}]$ which contradicts to the statement that $q_{n+1}$ is a closest recurrence time.
	
	Thus, $y_{a_{n-1}}=x_{(a_{n+1}-1)q_n+a_{n-1}q_{n-2}}$, which is only possible if for every $1<k\leq a_{n-1}$ holds $y_k=x_{(a_{n+1}-1)q_n+kq_{n-2}}$, that is, $x_{a_{n+1}q_n+kq_{n-2}}$ and $x_{(a_{n+1}-1)q_n+kq_{n-2}}$ have the same sign. This confirms the hypothesis for indices $(a_{n+1}-1)q_n+kq_{n-2}$.
	
	Using the inequality
	$$\abs{x_0}<\abs{x_{a_{n+1}q_n+a_{n-1}q_{n-2}}}<\abs{x_{(a_{n+1}-1)q_n+a_{n-1}q_{n-2}}}$$
	and applying $f$ in the usual way $m$ times we obtain for $0<m<q_{n-3}$
	$$x_{a_{n+1}q_n+a_{n-1}q_{n-2}+m}\in[x_m,x_{(a_{n+1}-1)q_n+a_{n-1}q_{n-2}+m}],$$
	which confirms the hypothesis for indices $a_{n+1}q_n+a_{n-1}q_{n-2}+m$.
	
	Now we are only left to determine signs of $x_{iq_n+q_{n-1}}$ with $0\leq i<a_{n+1}$. For this consider the inequalities
	$$\abs{x_0}<\abs{x_{q_n}}<\abs{x_{2q_n}}<\dots<\abs{x_{a_{n+1}q_n}},$$
	which are simply the collection of inequalities saying that the starting point of a block of $[a_1,a_2,...,a_n]$-recurrence is closer to $0$ than its ending point. For $0<m<q_{n-1}$, the $f^m$-images of all involved points have the same sign, which means that for $m=q_{n-1}$ their order on the real line is either the same as the order of the absolute values above if $\Neg(q_{n-1})$ is even, and inverted otherwise. Hence, since $x_{q_{n+1}}=x_{a_{n+1}q_n+q_{n-1}}$ is a point of closest recurrence, all points $x_{iq_n+q_{n-1}}$ must have the same sign: they are negative if the order is preserved, and positive otherwise. By Lemma~\ref{lmm:number_of_negative_terms} this encodes as follows: $x_{iq_n+q_{n-1}}$ is negative if $n-1\equiv 0,1\mod 4$, and positive otherwise.
	
	\textbf{Basis of induction.} The induction step from $n$ to $n+1$ uses information from steps $n,n-1$ if $a_n>1$ and information from steps $n,n-1,n-2,n-3$ if $a_n=1$. Therefore, it is enough to provide the basis for $n\leq 4$. Moreover, as shown in the paragraph before Definition~\ref{defn:admissible_angles}, we only need to consider those $\theta$ whose continued fraction $[a_1,a_2,a_3,a_4]$ is the beginning of a continued fraction of an admissible angle: $\theta=[1,1,1,a_4]$.
	
	In this case $q_1=1,q_2=2,q_3=3,q_4=3a_4+2$. By $\theta$-recurrence, $x_{iq_3+1}<0<x_{iq_3+2}$ for $0\leq i<a_4$ and we only need to determine the sign of $x_{a_4q_3+1}$. Since $\abs{x_0}<\abs{x_{a_4q_3}}<\abs{x_{(a_4-1)q_3+q_2}}$, we have $x_{a_4q_3+1}\in[x_1,x_{a_4q_3}]$. If $a_4=1$, then $a_4q_3=q_3$ is a time of closest recurrence, and therefore $x_{a_4q_3+1}=x_4<0$. If $a_4>1$, then we should take a look at the pair $x_2,x_{(a_4-1)q_3+q_2}$ satisfying $\abs{x_0}<x_{(a_4-1)q_3+2}<x_2$. Applying $f$, we get $x_{a_4q_3}<x_3$, and, as $x_3$ is a time of closest recurrence, $x_{a_4q_3}<0$. Hence, $x_{a_4q_3+1}<0$.
\end{proof}

To show existence of a $\theta$-recurrent unimodal map one can either apply the criterion from \cite{MT}, which also shows that such map is realized as a real quadratic polynomial, or construct it explicitly. We choose the latter option because it provides some additional information on the order of $\{x_k\}$ on the real line. 

\begin{thm}[Construction of a $\theta$-recurrent map]
	\label{thm:map_construction}
	For every admissible angle $\theta$ there is a $\theta$-recurrent unimodal map $f$ with critical orbit $\mathcal{O}=\{0,x_1,x_2,...\}$ such that every point in the closure $\overline{\mathcal{O}}$ is encoded by formal series (finite or infinite) $\varkappa=\sum_{i=1}^{\infty}\gamma_n q_n$ by $\varkappa\mapsto x_\varkappa$. The correspondence is a homeomorphism from all words in the number system associated to $\theta$ endowed with product topology, and satisfying $x_{1+\varkappa}=f(x_\varkappa)$.
\end{thm}
\begin{proof}
	An instance of such construction is sketched on Picture~\ref{pic:example}.
	
	First, we define a map $f$ only on $\mathcal{O}$, that is, we construct $\mathcal{O}$ in such a way that monotonicity on the left and on the right from $0$ is respected. Afterwards, we extend this $f$ to $\overline{\mathcal{O}}$ and to a unimodal map. 
	
	Denote $\mathcal{O}_n:=\{x_0,x_1,...,x_n\}$. We construct $\mathcal{O}$ inductively. Pick arbitrary real $x_1,x_2,x_3$ satisfying
	$$-1<x_1<x_0=0<x_3<x_2<-x_1.$$	
	Clearly, $f:\mathcal{O}_2\to\mathcal{O}_3$ satisfies the monotonicity property of a potential unimodal map. Assume that we have $\mathcal{O}_{q_n}$ with the property that the interval $(0,x_{q_{n-1}})$ is disjoint from $\mathcal{O}_{q_n}$ except possibly of $x_{q_n},x_{2q_n},...,x_{a_{n+1}q_n}$, the interval $(0,x_{q_n})$ is disjoint from $\mathcal{O}_{q_n}$, and construct $\mathcal{O}_{q_{n+1}}$.
	
	\begin{enumerate}
		\item Case $a_{n+1}=1.$ For every $0<m<q_{n-1}$, choose as $x_{m+q_n}$ any point in $I$ so that every interval $[x_m,x_{m+q_n}]$ has length less than $1/2^n$, intersects with $\mathcal{O}_{q_{n+1}-1}$ only at endpoints, and $x_m<x_{m+q_n}$ if $\Neg(m)$ is even and $x_m>x_{m+q_n}$ otherwise. Clearly, the monotonicity relation holds for $f:\mathcal{O}_{q_{n+1}-2}\to\mathcal{O}_{q_{n+1}-1}$.
		
		Pick as $x_{q_{n+1}}=x_{q_{n-1}+q_n}$ an arbitrary point such that the interval $(x_{q_{n+1}},0)$ is disjoint from $\mathcal{O}_{q_{n+1}}$, has length less than $1/2^n$, and the sign of $x_{q_{n+1}}$ is determined by $(1)$ of Theorem~\ref{thm:signs}. We need to show that monotonicity holds for $f:\mathcal{O}_{q_{n+1}-1}\to\mathcal{O}_{q_{n+1}}$. Note that the only points of $\mathcal{O}_{q_{n+1}}$ lying in the open interval $(x_{q_{n+1}},x_{q_{n-1}})$ are $0$ and possibly $x_{q_n}$. On the other hand, $(x_{q_{n-1}-1+q_n},x_{q_{n-1}-1})$ cannot contain $x_{q_n-1}$ by construction. Hence, we only need to check correctness of monotonicity for a pair $x_{q_{n-1}-1+q_n},x_{q_{n-1}-1}$. But since $\Neg(q_{n-1})$ is even when $x_{q_{n-1}}<0$, and odd otherwise, and also trivially $(-1)^{\Neg(q_{n-1}-1)}\sgn x_{q_{n-1}-1}=(-1)^{\Neg(q_{n-1})},$
		to satisfy the monotonicity, $x_{q_{n+1}}$ and $0$ must be on the same side of $x_{q_{n-1}}$, which is fulfilled by construction.
		
		\item Case $a_{n+1}>1.$ For every $0<m<q_n$, choose as $x_{m+q_n}$ any point in $I$ so that every interval $[x_m,x_{m+q_n}]$ has length less than $1/2^n$, intersects with $\mathcal{O}_{q_{n+1}-1}$ only at endpoints, and $x_m<x_{m+q_n}$ if $\Neg(m)$ is even and $x_m>x_{m+q_n}$ otherwise.The monotonicity relation holds for $f:\mathcal{O}_{2q_n-2}\to\mathcal{O}_{2q_n-1}$.
		
		Pick as $x_{2q_n}$ an arbitrary point such that the interval $(x_{2q_n},0)$ is disjoint from $\mathcal{O}_{2q_n}$, has length less than $1/2^n$, and the sign of $x_{2q_n}$ is determined by $(2)$ of Theorem~\ref{thm:signs}. We need to show that monotonicity holds for $f:\mathcal{O}_{2q_n-1}\to\mathcal{O}_{2q_n}$. The only point of $\mathcal{O}_{2q_n}$ lying in the open interval $(x_{2q_n},x_{q_n})$ is $0$. Hence, we only need to check correctness of monotonicity for pair $x_{2q_n-1},x_{q_n-1}$. But since $\Neg(q_n)$ is even when $x_{q_n}<0$, and odd otherwise, to satisfy the monotonicity, $x_{2q_n}$ and $0$ must be on the same side of $x_{q_n}$, which is fulfilled by construction.
		
		If $a_n>2$, we proceed. For every $0<m\leq(a_{n+1}-2)q_n$, choose as $x_{m+2q_n}$ any point in $I$ so that every interval $[x_{m+q_n},x_{m+2q_n}]$ has length less than $1/2^n$, intersects with $\mathcal{O}_{a_{n+1}q_n-1}$ only at endpoints, and $x_{m+q_n}<x_{m+2q_n}$ if $\Neg(m)$ is even and $x_{m+q_n}>x_{m+2q_n}$ otherwise. The monotonicity relation holds for $f:\mathcal{O}_{a_{n+1}q_n-1}\to\mathcal{O}_{a_{n+1}q_n}$. Note that for $0<k<q_n$, we have
		$$\Neg(k+q_n)\equiv\Neg(k+2q_n)\equiv...\equiv\Neg(k+(a_{n+1}-1)q_n)\mod 2$$
		hence, points in $\mathcal{O}_{a_{n+1}q_n}$ appear in form clusters of $a_{n+1}$ points in a row: $x_k,x_{k+q_n},...,x_{k+(a_{n+1}-1)q_n}$ where $0<k< q_n$. On the other hand, for $k=q_n$, point $x_{q_n}$ has opposite sign than a cluster of $a_{n+1}-1$ points in a row: $x_{2q_n},...,x_{a_{n+1}q_n}$. And these points are further from $0$ as the index is bigger.
		
		Now, we basically repeat the considerations from the case of $a_{n+1}=1$. For every $0<m<q_{n-1}$, choose as $x_{m+a_{n+1}q_n}$ any point in $I$ so that every interval $[x_{m+(a_{n+1}-1)q_n},x_{m+a_{n+1}q_n}]$ has length less than $1/2^n$, intersects with $\mathcal{O}_{q_{n+1}-1}$ only at endpoints, and $x_{m+(a_{n+1}-1)q_n}<x_{m+a_{n+1}q_n}$ if $\Neg(m)$ is even and $x_m+(a_{n+1}-1)q_n>x_{m+a_{n+1}q_n}$ otherwise. The monotonicity relation holds for $f:\mathcal{O}_{q_{n+1}-2}\to\mathcal{O}_{q_{n+1}-1}$.
		
		Pick as $x_{q_{n+1}}$ a point such that the interval $(x_{q_{n+1}},0)$ is disjoint from $\mathcal{O}_{q_{n+1}}$, has length less than $1/2^n$, and the sign of $x_{q_{n+1}}$ is determined by $(1)$ of Theorem~\ref{thm:signs}. We need to show that monotonicity holds for $f:\mathcal{O}_{q_{n+1}-1}\to\mathcal{O}_{q_{n+1}}$. The only points of $\mathcal{O}_{q_{n+1}}$ lying in the interval $(x_{q_{n+1}},x_{q_{n-1}+(a_{n+1}-1)q_n})$ are $0$ and possibly $x_{q_n},x_{2q_n},...,x_{a_{n+1}q_n}$. The interval $(x_{q_{n+1}}-1,x_{q_{n-1}-1+(a_{n+1}-1)q_n})$ cannot contain $x_{q_n-1},x_{2q_n-1},...,x_{a_{n+1}q_n-1}$ by construction. Hence, one should only check correctness of monotonicity for pair $(x_{q_{n+1}}-1,x_{q_{n-1}-1+(a_{n+1}-1)q_n})$. But since $\Neg(q_{n-1})$ is even when $x_{q_{n-1}}<0$, and odd otherwise,
		to satisfy the monotonicity, $x_{q_{n+1}}$ and $0$ must be on the same side of $x_{q_{n-1}}$, which is again fulfilled by construction.
	\end{enumerate}
	
	Thus, constructed $f:\mathcal{O}\to\mathcal{O}$ is decreasing on the left from $0$, and increasing on the right. Note that because of its definition, all limit points of $\mathcal{O}$ are limit points of sequences of finite words $[\gamma_1],[\gamma_1\gamma_2],...,[\gamma_1\gamma_2\cdots\gamma_n],...$, that is, they can be identified with infinite words $\varkappa=[\gamma_1\gamma_2\cdots]$. As a bijection from a compact into a Hausdorff space, the correspondence is homeomorphism.
	
	Finally, $f:\overline{\mathcal{O}}\to\overline{\mathcal{O}}$ extends to a unimodal map by linear interpolation. We only need to show that it is $\theta$-recurrent. But this follows from the construction: for $q_{n+1}=a_{n+1}q_n+q_{n-1}$, the orbits of length $q_n$ of points $0,x_{q_n},...,x_{(a_{n+1}-1)q_n}$ are split into clusters of $q_{n+1}$ points and comparison relations are the same for any point in the cluster. Hence, if we know, that $0$ is $[a_1,...,a_n]$-recurrent, it follows that $0$ is also $[a_1,...,a_{n+1}]$-recurrent.
\end{proof}

From the construction in Theorem~\ref{thm:map_construction} we can easily see how the points of $\mathcal{O}$ are placed on the real line:
\begin{enumerate}
	\item $x_1<0, x_2>0$ and all other points of $\mathcal{O}$ are between them;
	\item $\abs{x_{q_n}}>\abs{x_{q_{n+1}}}$;
	\item for $a_{n+1}>1$, $\abs{x_{q_{n-1}}}>\abs{x_{a_{n+1}q_n}}>\abs{x_{(a_{n+1}-1)q_n}}>\cdots>\abs{x_{q_n}}\to 0$;
	\item signs of the points above are controlled by the conditions in Theorem~\ref{lmm:signs};
	\item if $m=[\gamma_1\cdots\gamma_k\gamma_{k+1}\cdots]$, $n=[\gamma_1\cdots\gamma_k\gamma_{k+1}'\cdots]$ and $\gamma_{k+1}\neq\gamma_{k+1}'$, then $x_m>x_n$ iff $x_{[\gamma_1\cdots\gamma_k\gamma_{k+1}]}>x_{[\gamma_1\cdots\gamma_k\gamma_{k+1}']}$, that is to compare $x_m$ an $x_n$ it is enough to compare compare points corresponding to the minimal non-equal finite words of equal size in their Ostrowski presentation;
	\item for every $x_{kq_n}$, points $x_{kq_n+lq_m}, m>n$ converge to $x_{kq_n}$ monotonically as in item (3).
\end{enumerate}

We also need the usual notion of a \emph{kneading sequence} of $x_0$. This is a sequence $\{k_i\}_{i>0}$ satisfying $k_i=0$ if $x_i<0$ and $k_i=1$ if $x_i>0$. Due to Theorem~\ref{thm:signs} every irrational $\theta$ determines some kneading sequence. Due to Theorem~\ref{thm:map_construction}, it cannot be eventually periodic---this would imply that $0$ is periodic point. Therefore, every such kneading sequence defines a unique possible order of points $\mathcal{O}$ on the real line. That is, different admissible angles induce different kneading sequences.

As for Fibonacci map in \cite{LM}, every $\theta$-recurrent map with no \emph{homtervals} (intervals mapped homeomorphically by all iterates of the map) are topologically conjugate to the $\theta$-recurrent map constructed in Theorem~\ref{thm:map_construction}.

\begin{proof}[Proof of Theorem~\ref{thm:theta_recurrence}(2)]
	Note that $\theta>1/2$ because $a_1=1$. Denote $\theta_n:=q_n\alpha-p_n$. Construct the semiconjugacy $\varphi:\overline{\mathcal{O}}\to S^1$ using formula~\ref{eqn:ostrowski_reals} from Appendix as follows. Given a word $\varkappa=[\gamma_0\gamma_1\gamma_2\cdots]$, let 
	$$\varphi(x_\varkappa)=\sum_{k=0}^{\infty}\gamma_k\theta_k.$$
	The series above converges absolutely, so $\varphi$ is well defined and continuous. Clearly, $\varphi(f(x_\kappa))=\varphi(x_\kappa)-\theta$, which corresponds to rotation by the angle $-\theta$.
\end{proof}

Note, that this semiconjugacy is one-to-one except on the backward orbit of $0$.

Now, we provide a description of $\overline{\mathcal{O}}$, from which it is easy to see that $\overline{\mathcal{O}}$ is a Cantor set. It will often be used in the sequel.

Let $I_0^n$ be the smallest closed interval containing $x_{a_{n+1} q_n},x_{q_n},x_{q_{n+1}},x_{q_{n+2}}$. This means that if $a_{n+1}>1$, then $I^n_0=[x_{a_{n+1}q_n},x_{q_n}]$, and otherwise $I_0^n$ is equal to the biggest out of $[x_{q_n},x_{q_{n+1}}]$ and $[x_{q_n},x_{q_{n+2}}]$. Further, for $0<k<q_{n-1}$, let $I_k^n:=f^k(I_0^n)=[x_k,x_{k+a_{n+1}q_n}]$. From the relations in Theorem~\ref{thm:signs} one can see immediately the the restriction of $f$ to $I_k^n$ with $0<k<q_n$ are homeomorphisms, and $f(I_{q_{n-1}-1}^n)=[x_{q_{n-1}},x_{q_{n+1}}]\ni 0$. Next, for $q_{n-1}\leq k< q_n$ define intervals $J_k^n$ depending on the value of $a_{n+1}$ (note that here we changed the index range of $J_k^n$ comparing to \cite{LM} to obtain shorter indices):
\begin{itemize}
	\item $J_k^n:=[x_k, x_{k+(a_{n+1}-1)q_n}]$ if $a_{n+1}>1$,
	\item $J_k^n:=[x_k,x_{k+a_{n+2}q_{n+1}}]$ if $a_{n+1}=1$.
\end{itemize}
Again, the restriction of $f$ to any of $J_k^n$ is a homeomorphism and $f(J_{q_n-1}^n)$ is equal either to $[x_{q_n},x_{a_{n+1}q_n}]$ if $a_{n+1}>1$ or to $[x_{q_n},x_{q_{n+2}}]$ otherwise, but in both cases contains $0$.

Taking into account that all $I_k^n,J_l^n$ are mutually disjoint, define 
$$M^n:=(\bigcup_{k=0}^{q_{n-1}-1}I^n_k)\cup(\bigcup_{k=q_{n-1}}^{q_n-1}J_k^n).$$
We have
$$I_k^n=I_k^{n+1}\cup J_{k+q_n}^{n+1}\cup J_{k+2q_n}^{n+1}\cup\cdots\cup J_{k+a_{n+1}q_n}^{n+1}$$
and either
$$J_k^n=I_k^{n+1}\cup J_{k+q_n}^{n+1}\cup J_{k+2q_n}^{n+1}\cup\cdots\cup J_{k+(a_{n+1}-1)q_n}^{n+1}$$
if $a_{n+1}>1$, or
$$J_k^n=I_k^{n+1}$$
if $a_{n+1}=1$.

Thus, it is easy to see that $\overline{\mathcal{O}}=\bigcap M^n$ (there are no homtervals) and hence $\overline{\mathcal{O}}$ is a Cantor set.

We also extend the above definition of $J_k^n$ to the case $0<k<q_{n-1}$ and $k=q_n$. Clearly, each $f:J_k^n\to J_{k+1}^n$ is a homeomorphism.

We conclude this section by a technical lemma about combinatorics of $\theta$-recurrent maps. In the following section it will help us to pull-back certain intervals homeomorphically a ``maximal'' number of times. More precisely, it follows from the lemma that given two neighbouring intervals of $M^n$ containing (respectively) points $x_k,x_l$ with $k<l<q_n$, the convex hull of their union can be pulled-back homeomorfically $l$ times along the orbit $x_0,x_2,\dots,x_l$.

\begin{lmm}
	\label{lmm:neighbours}
	Let $f$ be a $\theta$-recurrent map. Fix some $n\in\mathbb{N}$ and two indices $i,j$ such that $0<j<i<q_n$. If for every $0<k\leq j$, the points $x_k$ and $x_{i-j+k}$ have the same sign, then the interval of $M^n$ containing $x_j$ is contained in $[x_i,x_j]$.
\end{lmm}
\begin{proof}
	
	Consider a pair of points corresponding to $k=1$, i.e., $x_1$ and $x_{i-j+1}$. The statement of the lemma is true for them since the other endpoint of the interval $I^n_1$ is closer to $x_1$ than any other point $x_l$ with $0<l<q_n$. The statement for $k>1$ follows immediately if we recall that $I_k^n$'s and $J_k^n$'s are obtained from each other by applying $f$ (or by shrinking to a subinterval and applying $f$ when changing from $I_k^n$'s to $J_k^n$'s).	  
\end{proof}

\section{A priori bounds}
\label{sec:a_priori}

The goal of this section is to generalize a priori estimates from \cite[Section~4]{LM} for $\theta$-recurrent maps using the same approach and machinery: Schwarz lemma and Koebe principle (see Appendix). We try to keep similar notation, statements  and flow of proofs. 

We work only with \emph{even} unimodal $C^2$ maps $f:[-1,1]\to[-1,1]$ such that $0$ is non-degenerate minimum point, $f(-1)=f(1)=1$, and $f$ coincides with a quadratic polynomial $x^2-c$ near $0$. This does not restrict the generality (see \cite[Section~4]{LM} for details).

First, we introduce some notation. For every pair $(i,n)$ such that $n>0$ and $a_{n+1}\geq i>0$ denote $d_n^i:=\absv{x_{iq_n}}$. Thus, $\{d_n^1\}_{n=1}^\infty$ are the magnitudes of closest returns and it holds
$$\dots<d_{n+1}^{a_{n+2}}<d_n^1<d_n^2<\dots<d_n^{a_{n+1}}<d_{n-1}^1<\dots$$
Further, denote
$$\delta_n^1:=\frac{d_n^1}{d_n^2},\; \delta_n^2:=\frac{d_n^2}{d_n^3},\;\dots,\;\delta_n^{a_{n+1}}:=\frac{d_n^{a_{n+1}}}{d_{n-1}^1}.$$
Note that $\delta_n^i<1$. Also we will occasionally use $\lambda_n:=d_n^1/d_{n-1}^1$. It is the asymptotic behaviour of $\delta_n^i$ that needs to be computed. However, this will be done in the next section. In the remaining part of the current section we provide a number of a priory bounds for them.

We say that the intervals $\mathbf{G}=\{G_i\}_{i=0}^k$ form a \emph{chain of intervals} if each $G_i$ is a connected component of $f^{-1}G_{i+1}$. The chain is \emph{monotone} if every $f:G_i\to G_{i+1}$ is a homeomorphism.

Next, for a family of intervals $\mathbf{G}=\{G_i\}$ let $\absv{\mathbf{G}}:=\sum_{i}\absv{G_i}$ be the measure of $\mathbf{G}$ and $\mult\mathbf{G}$ be the maximum number of $G_i$ whose intersection is non-empty.

Denote $T^n:=[x_{q_n},x_{q_n}']$ and let $\mathbf{H}^n=\{H_i^n\}_{i=1}^{q_n}$ be the pull-back of $H_{q_n}^n:=T^{n-2}$ along the orbit $\{x_i\}_{i=1}^{q_n}$. Note that unlike \cite{LM} we prefer notation $\mathbf{H}^n$ rather than $\mathbf{H}^{n+1}$ for this set (it simply fits better for the generalized case).

\begin{lmm}
	The chain $\mathbf{H}^n$ is monotone and $J_1^n\subset H_1^n$.
\end{lmm}
\begin{proof}
	If $\mathbf{H}^n$ is not monotone, then the interior of one of $H_i^n$ with $i<q_n$ must contain $0$ which maps along $\mathbf{H}^n$ to the interior of $T^{n-2}$. The post-critical points with indices less than $q_n$ contained in the interior of $T^{n-2}$ are exactly 
	$$x_{q_{n-1}},\dots,x_{a_n q_{n-1}}.$$	
	It is enough to show that for any of these points, say $x_k$, the orbit $\{f^i([x_{q_n-k},0])\}_{i=0}^k$ is not monotone. This follows easily from combinatorics (Theorem~\ref{thm:signs}).
	
	The second statement holds because $f(J_{q_n-1}^n)\subset T^{n-2}$. 
\end{proof}

For any maximal subinterval $I$ of $M^n$, except $I_1^n$ and $I_2^n$, denote by $F(I)$ the minimal interval containing $I$ and two of its neighbours in $M^n$. For any such $I$ (containing $x_k$ with $k<q_n$), denote by $\mathbf{G}=\{G_i\}_{i=0}^k$ denote the pull-back of $G_k:=F(I)$ along $\{x_i\}_{i=0}^k$.

\begin{lmm}
	\label{lmm:monotonicity_of_G}
	The chain $\{G_i\}_{i=1}^k$ is monotone and, depending on the value of $k$, the following statements hold:
	\begin{itemize}
		\item ($k<q_{n-1}$) $G_0\subset T^{n-1}$;
		\item ($k\geq q_{n-1}$) $G_0\subset [x_{a_{n+1}q_n},x_{a_{n+1}q_n}']$ if $a_{n+1}>1$, and $G_0\subset [x_{a_{n+2}q_{n+1}},x_{a_{n+2}q_{n+1}}']$ otherwise.
	\end{itemize}
\end{lmm}
\begin{proof}
	Monotonicity follows from lemma~\ref{lmm:neighbours}.
	
	Suppose $k<q_{n-1}$. Since $x_{1+q_{n-1}}$ is the closest to $x_1$ point among $\{x_i\}_{i=0}^{q_n}$, one has $G_0\subset T^{n-1}$.
	
	If $k\geq q_{n-1}$, it is enough to notice $0\in f^{q_{n-1}}([x_{a_{n+1}q_n},x_{a_{n+1}q_n}'])$. By monotonicity of $\{G_i\}_{i=1}^k$, one has $G_0\subset [x_{a_{n+1}q_n},x_{a_{n+1}q_n}']$.
\end{proof}

Proof of the next lemma goes exactly as in \cite[Lemma~4.3]{LM}.

\begin{lmm}
	\label{lmm:multiplicity}
	$\mult\mathbf{H}^n<C_1 a_n$ and $\mult\mathbf{G}<C_2$ for some universal constants $C_1,C_2$.
\end{lmm}

The following few lemmas provide the desired a priori bounds.

\begin{lmm}
	\label{lmm:a_priori_1}
	There exists a constant $C<1$ depending only on $f$ such that for every index $n$:
	\begin{itemize}
		\item ($a_{n+1}>1$) $\delta_n^{a_{n+1}-1}\delta_n^{a_{n+1}}<C$,
		\item ($a_{n+1}=1$) $\delta_{n+1}^{a_{n+2}}\delta_n^{1}<C$.
	\end{itemize}
\end{lmm}
\begin{proof}
	The proof is identical the one of \cite[Lemma~4.4]{LM} with application of lemmas~\ref{lmm:monotonicity_of_G} and \ref{lmm:multiplicity} and replacement of $\lambda_{n+1},\lambda_n$ by $\delta_n^{a_{n+1}-1},\delta_n^{a_{n+1}}$ or $\delta_{n+1}^{a_{n+2}},\delta_n^{1}$, respectively. 
\end{proof}

\begin{lmm}
	\label{lmm:a_priori_2}
	Depending on the value of $a_{n+1}$, we have the following inequalities.
	\begin{itemize}
		\item ($a_{n+1}>1$)
		$$\frac{1}{1-(\delta_{n+1}^{a_{n+2}})^2}\leq\left(\frac{1+\delta_n^1}{1-\delta_n^1}\right)^2+O(a_n \absv{\mathbf{H}^n}),$$
		$$\frac{1}{1-(\delta_n^1)^2}\leq\left(\frac{1+\delta_n^2}{1-\delta_n^2}\right)^2+O(a_n \absv{\mathbf{H}^n}),$$
		$$\cdots,$$
		$$\frac{1}{1-(\delta_n^{a_{n+1}-2})^2}\leq\left(\frac{1+\delta_n^{a_{n+1}-1}}{1-\delta_n^{a_{n+1}-1}}\right)^2+O(a_n \absv{\mathbf{H}^n}),$$
		$$\frac{1}{1-(\delta_n^{a_{n+1}-1})^2}\leq\left(\frac{1+\delta_n^{a_{n+1}}\lambda_{n-1}}{1-\delta_n^{a_{n+1}}\lambda_{n-1}}\right)^2+O(a_n \absv{\mathbf{H}^n}).$$
		\item ($a_{n+1}=1$)
		$$\frac{1}{1-(\delta_{n+1}^{a_{n+2}})^2}\leq\left(\frac{1+\delta_n^1\lambda_{n-1}}{1-\delta_n^1\lambda_{n-1}}\right)^2+O(a_n \absv{\mathbf{H}^n}).$$
	\end{itemize}
	Constants in $O(\cdot)$ depend only on $f$. If $f$ has non-positive Schwarzian derivative, then $O(\cdot)$ is equal to $0$.
\end{lmm}
\begin{proof}
	($a_{n+1}>1$) To prove the first inequality one has to consider the monotone map
	$$f^{q_n-1}:\left([x_1,x_{1+q_n}],[x_1,x_{1+a_{n+2}q_{n+1}}]\right)\to\left([x_{2q_n},x_{2q_n}'],T^n\right),$$
	and to literally repeat computations of \cite[Lemma~4.5]{LM}
	
	Analogously, for every $i=1,\dots,a_{n+1}-2$, deal with the map
	$$f^{q_n-1}:\left([x_1,x_{1+(i+1)q_n}],[x_1,x_{1+iq_n}]\right)\to\left([x_{(i+2)q_n},x_{(i+2)q_n}'],[x_{(i+1)q_n},x_{(i+1)q_n}']\right).$$
	
	The last inequality corresponds to the map
	$$f^{q_n-1}:\left(H_1^n,[x_1,x_{1+(a_{n+1}-1)q_n}]\right)\to\left(T^{n-2},[x_{a_n q_n},x_{a_n q_n}']\right).$$
	One only needs to note that $f^{-1}(H_1^n)\subset [x_{a_{n+1}q_n},x_{a_{n+1}q_n}']$ and repeat the same computation.
	
	($a_{n+1}=1$) In this case one needs to consider
	$$f^{q_n-1}:\left(H_1^n,[x_1,x_{1+a_{n+2}q_{n+1}}]\right)\to\left(T^{n-2},[x_{q_n},x_{q_n}']\right),$$
	and note that $f^{-1}(H_1^n)\subset T^n$.
\end{proof}

\begin{lmm}
	\label{lmm:a_priori_bounds_deltas}
	Let $f$ be a $\theta$-recurrent map. If the continued fraction of $\theta$ has bounded denominators, then there exists a finite sequence of constants $\{C_i\}, C_i<1$, depending only on $f$, such that $\delta_n^i<C_i$ for every pair $i,n$ such that $i\leq a_{n+1}$.
\end{lmm}
\begin{proof}
	From lemma~\ref{lmm:a_priori_2} one can see that if $\delta_n^i$ with fixed $i$ can be arbitrarily close to $1$, then so is $\delta_n^{i+1}$ and analogously with the pair $\delta_{n+1}^{a_{n+2}},\delta_n^1$. This contradicts to lemma~\ref{lmm:a_priori_1}.
\end{proof}

Next, we derive estimates for measures of $M^n$ and $\mathbf{H}^n$ which will prove useful in the next section.

\begin{lmm}
	\label{lmm:area_bounds}
	If $\theta$ has bounded denominators, then there exist constants $C>0$ and $0<p<1$ such that for all $n$,
	$$\absv{M^n}<Cp^n \text{ and  } \absv{\mathbf{H}^n}<Cp^n.$$
\end{lmm} 
\begin{proof}
	
	As earlier in this section, the discussion is analogous to the one of \cite[Lemma~4.8]{LM}. 
	
	First, we want to show that the values $\absv{M^n\cap I^{n-1}_0}/\absv{I^{n-1}_0}$ are bounded from $1$.
	
	Consider the case $a_n>1$. Let $L^n$ be the gap between $[x_{(a_n-1)q_{n-1}},x_{(a_n-1)q_{n-1}}']$ and $J_{a_n q_{n-1}}^n$. Due to Lemma~\ref{lmm:a_priori_bounds_deltas} it is enough to show that its length cannot be arbitrarily small compared to $\absv{J_{a_n q_{n-1}}^n}$. Note that the restriction $f^{q_{n-2}}|_{L^n\cup J_{a_n q_{n-1}}^n}$ is a homeomorphism and $f^{q_{n-2}}(L^n\cup J_{a_n q_{n-1}}^n)\subset T^{n-2}$. Also, depending on whether $a_{n+1}$ is either bigger or equal to one, we get $f^{q_{n-2}}J_{a_n q_{n-1}}^n$ equal to either $[x_{a_{n+1}q_n},x_{q_n}]$ or $[x_{q_{n+2}},x_{q_n}]$. Let $U^n:=T^{n-1} \cup f^{q_{n-2}}(L^n\cup J_{a_n q_{n-1}}^n)$. From Lemma~\ref{lmm:a_priori_bounds_deltas} we see that the Poincar\'e length $[f^{q_{n-2}}J_{a_n q_{n-1}}^n:T^{n-1}]$ is bounded from above. Pulling back these two intervals $q_{n-2}$ times homeomorphically along with $J_k^n$ and using Schwarz lemma we get that $L^n$ cannot be too small with respect to $J_{a_n q_{n-1}}^n$.
	
	If $a_n=1$, denote by $L^n$ the gap between $T^n$ and $J_{q_{n-1}}^n$ and repeat the same considerations.
	
	Now, the bounds for $M^n$ are proved as follows. Due to Lemma~\ref{lmm:a_priori_bounds_deltas} and Koebe principle all maps $f^{q_{n-1}-k}:I_k^n\to [x_{q_{n-1}},x_{q_{n+1}}]\subset T^{n-3},0<k<q_{n-1}$ and $f^{q_n-k}:J_k^n\to J_{q_n}^n\subset T^{n-2},q_{n-1}\leq k<q_n$ have uniformly bounded distortion on its domain of definition. Thus, $\absv{M^{n+1}}/\absv{M^n}$ is bounded from $1$ which proves the statement for $M^n$.
	
	Finally, we derive the bounds for $\mathbf{H}^n$. Denote by $R=R^{n+1}$ the convex hull of intervals $H_1^{n+1},H_{1+q_n}^{n+1},\dots,H_{1+(a_{n+1}-1)q_n}^{n+1}$. 
	Clearly, $f^{q_n}$ is monotonic on $R$. Moreover, $R\subset H_1^n$. Indeed, $f^{q_n}H_1^n=T^{n-2}$ and $f^{q_n}R\subset [x_{a_nq_{n-1}},x_{a_nq_{n-1}}']$ because $f^{q_{n-1}-1}|_{[x_{1+lq_n},x_{1+a_nq_{n-1}}]},0<l\leq a_{n+1}$ is not monotonic. Since the quantity $\left[[x_{a_nq_{n-1}},x_{a_nq_{n-1}}']:T^{n-2}\right]$ is bounded from $1$, from Schwarz lemma we get that $\absv{f^k R}/\absv{H^n_k}$ is uniformly bounded from $1$. Further, since $H_{1+a_{n+1}q_n}^{n+1}\subset I^{n-1}_1$, it is enough to estimate the measure of $f^k I^{n-1}_0, 0<k<q_{n-1}$. From combinatorics we have $f^{i+lq_{n-2}}I^{n-1}_0\subset I^{n-1}_i,0\leq l<a_n,0<i<q_{n-2}$ and $f^{lq_{n-2}}I^{n-1}_0\subset I^{n-2}_0,0<l\leq a_n$.
	Hence, $\cup_{k=1}^{q_{n-1}}f^k I_0^{n-1}\subset M^{n-1}\cup M^{n-2}$ and $\absv{\mathbf{H}^{n+1}}<p\absv{\mathbf{H}^n}+\absv{M^{n-1}}+\absv{M^{n-2}}$. The claim for $\mathbf{H}^n$ follows.	
\end{proof}

\section{Asymptotics and Hausdorff dimension}
\label{sec:Hausdorff}

We are now ready to prove an asymptotic formula connecting different $\delta_n^i$, estimate their magnitudes and prove that the Hausdorff dimension of the post-critical set is equal to $0$ (Compare with \cite[Section~5]{LM}).

We begin with estimates on derivatives. Recall that $f$ is quadratic near the origin.

\begin{lmm}
	\label{lmm:derivatives}
	Let $f$ be $\theta$-recurrent for some $\theta$ with bounded denominators. The following estimates take place.
	
	For every $x\in[x_1,x_{1+a_{n+2}q_{n+1}}]$,
	$$\absv{(f^{q_n-1})'(x)}=\frac{d_n^1}{(d_{n+1}^{a_{n+2}})^2}\left(1+O(\lambda_{n+1}+\lambda_{n-1}+q^n)\right).$$
	
	If $a_{n+1}>1$, then for $0<i<a_{n+1}$ and $x\in [x_1,x_{1+iq_n}]$,
	$$\absv{(f^{q_n-1})'(x)}=\frac{d_n^{i+1}}{(d_n^i)^2}\left(1+O(\delta_n^i+\lambda_{n-1}+q^n)\right).$$
	The constants $q<1$ and $O(.)$ do not depend on $n$.
	
\end{lmm}
\begin{proof}
	
	Let $x\in[x_1,x_{1+a_{n+2}q_{n+1}}]$. Consider the homeomorphism
	$$f^{q_n-1}:\left(H_1^n,[x_1,x_{1+a_{n+2}q_{n+1}}]\right)\to \left(T^{n-2},[x_{q_n},x_{q_{n+2}}]\right).$$
	From Koebe principle (for $f^{-1}$) and bounds in Lemma~\ref{lmm:area_bounds} for any $\xi\in[x,x_{1+a_{n+2}q_{n+1}}]$ we have
	$$\frac{(f^{q_n-1})'(x)}{(f^{q_n-1})'(\xi)}=1+O\left(\left[T^n:T^{n-2}\right]+q^n\right)=1+O\left(\lambda_{n-1}+q^n\right),$$
	where the constants do not depend on $n$. By Mean Value Theorem, there exists $\xi$ such that
	$$\absv{(f^{q_n-1})'(\xi)}=\frac{d_n^1+d_{n+2}^1}{(d_{n+1}^{a_{n+2}})^2}=\frac{d_n^1}{(d_{n+1}^{a_{n+2}})^2}\left(1+\lambda_{n+2}\lambda_{n+1}\right).$$
	Thus, we obtain the formula
	$$\absv{(f^{q_n-1})'(x)}=\frac{d_n^1}{(d_{n+1}^{a_{n+2}})^2}\left(1+O(\lambda_{n+1}+\lambda_{n-1}+q^n)\right).$$
	
	Next, if $a_{n+1}>1$, for $0<i<a_{n+1}$ consider the homeomorphism
	$$f^{q_n-1}:\left(H_1^n,[x_1,x_{1+iq_n}]\right)\to \left(T^{n-2},[x_{q_n},x_{(i+1)q_n}]\right).$$
	Repeating the argument above we get
	$$\frac{(f^{q_n-1})'(x)}{(f^{q_n-1})'(\xi)}=1+O\left(\left[[x_{q_n},x_{(i+1)q_n}]:T^{n-2}\right]+q^n\right)=1+O\left(\lambda_{n-1}+q^n\right)$$
	and
	$$\absv{(f^{q_n-1})'(\xi)}=\frac{d_n^1+d_n^{i+1}}{(d_n^i)^2}=\frac{d_n^{i+1}}{(d_n^i)^2}\left(1+\delta_n^1\delta_n^2\cdots\delta_n^i\right).$$
	Hence,
	$$\absv{(f^{q_n-1})'(x)}=\frac{d_n^{i+1}}{(d_n^i)^2}\left(1+O(\delta_n^i+\lambda_{n-1}+q^n)\right).$$
\end{proof}

Now, we can describe the asymptotic behaviour of $\delta_n^i$. To simplify the notation denote $\alpha_n:=\delta_{n+1}^{a_{n+2}}$.

\begin{lmm}[Asymptotics]
	\label{lmm:asymptotics}
	Assume that $\theta$ has bounded denominators. There exist constants $\epsilon,N>0$ depending only on $f$ (but not on $\theta$) such that the following statement holds.
	
	If for some $n>N$ any $\delta_n^i<\epsilon$, then the sequence $\{\delta_n^{a_{n+1}}\}$ decreases at least with the exponential speed and the following asymptotic formulas hold as $n\to\infty$.
	
	$$\left(\delta_{n+2}^{a_{n+3}}\right)^{2^{a_{n+2}}}\sim\left(\delta_{n+1}^{a_{n+2}}\right)^{2^{a_{n+1}}-1}\delta_n^{a_{n+1}},$$
	and for $0<i<a_{n+1}$,
	$$\delta_n^i\sim\left(\delta_{n+1}^{a_{n+2}}\right)^{2^i}.$$	
\end{lmm}
\begin{proof}
	From estimates in Lemmas~\ref{lmm:a_priori_2},~\ref{lmm:area_bounds} and boundedness of denominators of $\theta$ follows immediately that if some $\delta_n^i$ for a big enough $n$ is small, so are $\delta_m^j$ with $m\in\{n+1,n+2,n+3\}$ and $\delta_n^k$ with $k\leq i$.
	
	If $a_{n+1}>1$, from Lemma~\ref{lmm:derivatives}
	$$\frac{d_n^1}{(d_{n+1}^{a_{n+2}})^2}\approx\frac{d_n^{2}}{(d_n^1)^2}\approx\frac{d_n^{3}}{(d_n^2)^2}\approx\cdots\approx\frac{d_n^{a_{n+1}}}{(d_n^{a_{n+1}-1})^2},$$
	where the sign $\approx$ is understood in equality modulo a factor close to $1$ (of course, at this step it does not have asymptotic form). This is equivalent to 
	$$\left(\frac{d_n^1}{d_{n+1}^{a_{n+2}}}\right)^2\approx\frac{d_n^{2}}{d_n^1}, \left(\frac{d_n^{2}}{d_n^1}\right)^2\approx\frac{d_n^{3}}{d_n^2},\cdots,\left(\frac{d_n^{a_{n+1}-1}}{d_n^{a_{n+1}-2}}\right)^2\approx\frac{d_n^{a_{n+1}}}{d_n^{a_{n+1}-1}},$$
	or
	$$\left(\delta_{n+1}^{a_{n+2}}\right)^2\approx\delta_n^1,\left(\delta_n^1\right)^2\approx\delta_n^2,...,\left(\delta_n^{a_{n+1}-2}\right)^2\approx\delta_n^{a_{n+1}-1}.$$
	
	This turns into
	$$\delta_n^i\approx\left(\delta_{n+1}^{a_{n+2}}\right)^{2^i}$$
	for $0<i<a_{n+1}$.
	
	Thus, $\delta_n^i=\left(\alpha_n\right)^{2^i}$ and $$\lambda_n=\alpha_n^{2^1+2^2+\dots+2^{a_{n+1}-1}}\alpha_{n-1}=\alpha_n^{2^{a_{n+1}}-2}\alpha_{n-1}.$$
	
	Next, we substitute formulas from Lemma~\ref{lmm:asymptotics} into the expression
	$$(f^{q_{n+1}-1})'(x_1)=(f^{q_n-1})'(x_1)\cdot2x_{q_n}\cdot(f^{q_n-1})'(x_{1+q_n})\cdot2x_{2q_n}\times\cdots$$
	$$\times(f^{q_n-1})'(x_{1+(a_{n+1}-1)q_n})\cdot2x_{a_{n+1}q_n}\cdot(f^{q_{n-1}-1})'(x_{1+a_{n+1}q_n})$$
	and obtain
	$$\frac{d_{n+1}^1}{(d_{n+2}^{a_{n+3}})^2}\approx\frac{d_n^1}{(d_{n+1}^{a_{n+2}})^2}\cdot2d_n^1\cdot\frac{d_n^{2}}{(d_n^1)^2}\cdot2d_n^2\times\cdots$$
	$$\times\frac{d_n^{a_{n+1}}}{(d_n^{a_{n+1}-1})^2}\cdot2d_n^{a_{n+1}}\cdot\frac{d_{n-1}^1}{(d_n^{a_{n+1}})^2}$$
	or
	$$(\delta_{n+2}^{a_{n+3}})^2 \lambda_{n+1}\approx\frac{1}{2^{a_{n+1}}}\left(\delta_{n+1}^{a_{n+2}}\right)^2\lambda_n.$$
	
	Substituting formulas for $\alpha_n$'s we see
	\begin{equation}
		\label{eqn:asymptotics}
		\alpha_{n+1}^{2^{a_{n+2}}}\alpha_n\approx\frac{1}{2^{a_{n+1}}}\alpha_n^{2^{a_{n+1}}}\alpha_{n-1},
	\end{equation}
	that is, the values of $\alpha_n^{2^{a_{n+1}}}\alpha_{n-1}$ change by ``scaling''. 
	
	Clearly, this and the above formulas turn into asymptotic formulas if we prove that $\alpha_n$ tend to zero. Thus, to finish proof of the lemma we only have to show convergence with exponential speed.
	
	Let $K$ be equal to the maximum of $\{a_k\}_{k=1}^\infty$. First, assume that $a_{n+1}=K$,$n$ is big and $\alpha_{n-1},\alpha_n$ are small. From formula~\ref{eqn:asymptotics} applied $K$ times follows that $$\min\{\alpha_{n+k},\alpha_{n+k-1}\}<\gamma^k\max\{\alpha_n,\alpha_{n-1}\}$$ for some constant $\gamma<(3/4)^{1/K}$. Note, that from estimates of Lemma~\ref{lmm:a_priori_2} follows that for small $\alpha_k$ with big $k$ holds $\alpha_{k+1}=O(\alpha_k^{1/2^{a_{n+2}}}+p^k)$ with constants $O(.),p$ depending only on $K$. So, unless $\alpha_{n+k}$ is bigger than both $\alpha_{n+k-1},\alpha_{n+k+1}$, we obtain 
	$$\alpha_{n+k}<\gamma^l\max\{\alpha_n,\alpha_{n-1}\}$$
	where $l=k$ or $l=k+1$. Otherwise, which cannot happen twice in a row, we can apply the estimate $\alpha_{n+k}=O(\alpha_{n+k-1}^{1/2^{a_{n+k+1}}}+p^{n+k-1})$, and convergence to $0$ with exponential speed follows.
	
	Next, assume that $a_{n+1}$ is maximal among $a_{n+1},a_{n+2},\dots,a_{n+N}$. By the reasoning a in the paragraph above, if $n$ is big, $\alpha_{n-1},\alpha_n$ are small and $N$ is big enough the finite sequence $\alpha_n,\alpha_{n+1},...,\alpha_{n+N}$ tends to $0$ exponentially fast with the same exponent, and $N=N(a_{n+1})$ depends only on the value $a_{n+1}$. This way are defined $K$ numbers $N(1),N(2),\dots,N(K)$. Choose the estimate for $\alpha_n$ so that if $\alpha$ is small, then $\alpha_{n+K \max N(i)}$ is still small. This is possible due to Lemma~\ref{lmm:a_priori_2}. So, either our sequence exponentially decreases until some $n+k$ with $k>N(a_{n+1}), a_{n+k+1}>a_{n+1}$, or $k\leq N(a_{n+1})$. In both cases we can consider as ``base point'' $\alpha_{n+k}$ (which is small). Repeating this at most $K-1$ times we are in the case $a_n=K$. This finishes the proof of the lemma.
\end{proof}

In the setting of the previous lemma the lengths $d_n^i$ decrease with the superexponential speed. Indeed, from Lemma~\ref{lmm:asymptotics} we have for $0<i<a_{n+1}$ $$d_n^i/d_1^1=(\delta_n^i\delta_n^{i+1}\cdot...\cdot\delta_n^{a_{n+1}})\lambda_{n-1}\lambda_{n-2}\cdot\dots\cdots\lambda_2=$$
$$(\alpha_n^{2^i}\cdot\dots\cdot\alpha_n^{2^{a_{n+1}-1}}\alpha_{n-1})(\alpha_{n-1}^2\cdot\dots\cdot\alpha_{n-1}^{2^{a_n-1}}\alpha_{n-2})\cdot\dots\cdots\alpha_2\alpha_1<$$
$$\left(p^n\right)^{2^{a_{n+1}}-2^i}\left(p^{n-1}\right)^{2^{a_n}-1}\cdot\dots\cdot(p^1)^{2^{a_2}-1}C_1=p^{n(2^{a_{n+1}}-2^i)+(n-1)(2^{a_n}-1)+\dots+(2^{a_2}-1)}C_1$$
for some constant $p<1$. Analogously, if $a_{n+1}=1$,
$$d_n^1/d_1^1=p^{n+(n-1)(2^{a_n}-1)+\dots+(2^{a_2}-1)}C_1.$$
Hence, $d_n^i<C p^{n^2+n} C$.

Define by $\mathcal{F}_0(\theta)$ the class of $\theta$-recurrent unimodal maps such that $\delta_n^1$ attain arbitrarily small values. From Lemma~\ref{lmm:asymptotics} follows that if $\theta$ has bounded denominators and $f\in\mathcal{F}_0(\theta)$, the $d_n^i$ decrease with the superexponential speed and the asymptotics of Lemma~\ref{lmm:asymptotics} take place. This is enough to estimate the Hausdorff dimension of $\overline{\mathcal{O}}$.

\begin{lmm}
	\label{lmm:dim=0}
	If $f\in\mathcal{F}_0(\theta)$, the Hausdorff dimension of $\overline{\mathcal{O}}$ is equal to $0$.
\end{lmm}
\begin{proof}
	As for the Fibonacci case, it is enough to compute the Hausdorff $\varepsilon$-measure of every $M^n$ for every $\alpha>0$. The approach is the same as in \cite[Lemma~5.5]{LM}: first, show that $\absv{M^{n+1}\cap I^{n}_0}/\absv{I^n_0}$ decreases at least exponentially, and then due to bounds on distortion of $f^{q_n}$ the lengths of corresponding iterates of $M^{n+1}\cap I^{n}_0\subset I_0^n$ decrease with comparable speed. 
	
	Denote by $M^n(x)$ the interval of $M^n$ containing $x$. The set $M^{n+1}(0)=I^{n+1}_0$ and, if defined, all sets $M^{n+1}(x_{iq_n}),0<i<a_{n+1}$ are contained in $[x_{(a_{n+1}-1)q_n},x_{(a_{n+1}-1)q_n}']$ and $\absv{[0,x_{(a_{n+1}-1)q_n}]}/\absv{I^n_0}<q^n$. There is only one interval of $M^{n+1}\cap M^n$ outside of $[x_{(a_{n+1}-1)q_n},x_{(a_{n+1}-1)q_n}']$: $M^{n+1}(x_{q_n})$ equal to either $[x_{a_{n+1}q_n},x_{a_{n+1}q_n+(a_{n+2}-1)q_{n+1}}]$ if $a_{n+2}>1$, or $[x_{a_{n+1}q_n},x_{a_{n+1}q_n+a_{n+3}q_{n+2}}]$ if $a_{n+2}=1$. So, $f^{q_{n-1}}$ maps $M^{n+1}(x_{q_n})$ to either $[x_{q_{n+1}},x_{a_{n+2}q_{n+1}}]$, or $[x_{q_{n+1}},x_{q_{n+3}}]$, while $[0,x_{a_{n+1}q_n}]$ is mapped homeomorphically to $[x_{q_{n-1}},x_{q_{n+1}}]$; that is, the ratio of two images is lass than $p^n$. Since the distortion of $f^{q_{n-1}-1}$ on $[x_1,x_{1+a_{n+1}q_n}]$ is close to $0$, and near the origin $f$ is quadratic, we obtain that $\absv{M^{n+1}(x_{q_n})}/\absv{I^n_0}<p_1^n$.
	
	Further, because $f^{q_{n-1}-k}:{I_k^n}\to T^{n-1}$ and $f^{q_n-k}:{J_k^n}\to T^n$ have bounded distortion, we obtain that the ratio of lengths of every $M^{n+1}(x)$ outside of $I^{n+1}_0$ to the corresponding $M^n(x)$ is less than $p_2^n$.
	
	The Hausdorff $\varepsilon$-dimension of $M^n$ is less than
	$$\sum_{k=0}^{q_n-1} M^n(x_{k})^\varepsilon<q_n p_2^{((n-1)+(n-2)+\dots+2+1)\varepsilon}C=q_n p_2^{n^2 \varepsilon}C\to 0$$
	as $n\to\infty$ because $q_n$ grow at most exponentially. Hence, the Hausdorff dimension of $\overline{\mathcal{O}}$ is equal to $0$.
\end{proof}

\section{Circle renormalizations for unimodal maps}
\label{sec:renorm}

As was shown earlier, the $\theta$-recurrent maps are not infinitely renormalizable. However, it possible to make a surgery, not affecting the dynamics of the critical value, and to obtain a function from the class $\mathcal{A}$, defined below, and an associate (i.e.\ dependent on $\theta$) renormalization procedure within this class. In case of $\theta$ of bounded type it will follow that if $\delta_n^i$'s are bounded from below, the sequence of renormalizations must have a limiting point which is ``almost'' polynomial-like of type (2,1) (see \cite{LM}). This, however, leads to the same contradiction as in case of Fibonacci maps.  

We introduce a special class of functions, generalizing class $\mathcal{A}$ from \cite{LM}, which will allow to define a renormalization operator preserving this class. Let $J,T\subset I=[-1,1]$, be disjoint closed intervals and $0$ belongs to the interior of $T$. Consider a function 
$$f:J\cup T\to I$$
such that
\begin{enumerate}
	\item $f:J\to I$ is homeomorphism,
	\item $f:T\to I$ is unimodal with the minimum point at $0$ and $f(\partial T)=\{1\}$.
\end{enumerate}
Space of all such functions we denote by $\mathcal{A}$. Note, that comparing to \cite{LM}, our definition does \emph{not} fix $J$ to be on the left of $T$. 
\begin{figure}[h]
	\includegraphics[width=20em]{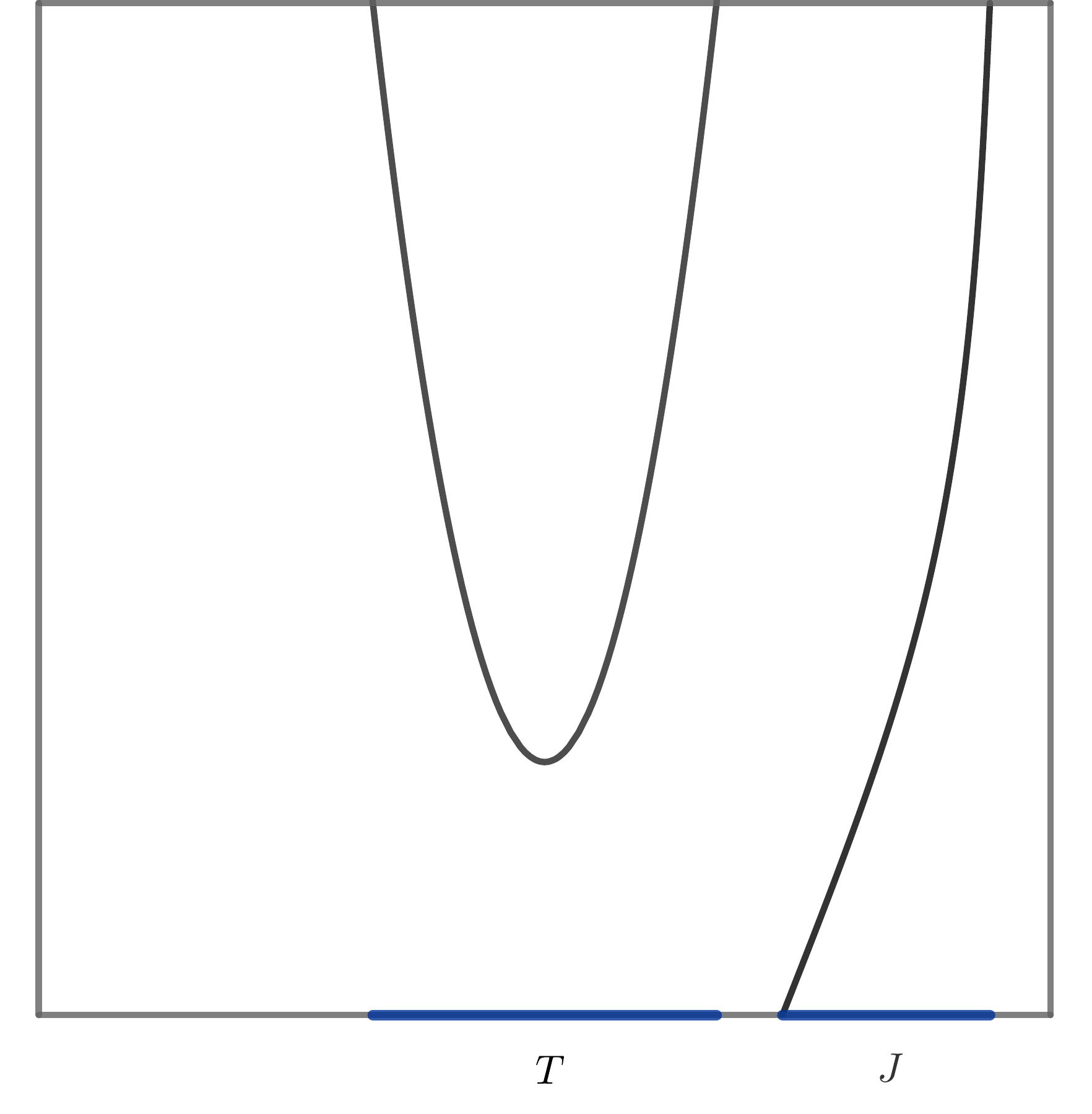}
	\caption{A function belonging to $\mathcal{A}$.}
	\label{pic:a-function}
\end{figure}

From here we assume that $T$ is symmetric with respect to $0$ and $f|T$ is even. Before providing a formal definition of the renormalization operator, we need to introduce more precise subclasses of $\mathcal{A}$. Let $\theta=[a_1,a_2,a_3,...]\in (0,1)$ be an irrational angle (not necessarily admissible), and $s\in\{+,-\}$. We are going to describe a class $\mathcal{A}_\theta^s$ in a similar way to how it is done in Theorem~\ref{thm:map_construction} by describing the position of the critical orbit on $I$, but for a function from class $\mathcal{A}$. Here is the description. 
\begin{enumerate}
	\item $0$ has infinite well-defined orbit $\{x_n\}$ and times of closest recurrence of $0$ are $q_0=1$ (if $a_1>1$), $q_1,q_2,...$. Denote by $q_b=1$ the time of the first best recurrence, i.e., either $b=0$ or $b=1$;
	\item if $q_{b+1}=2$, then $J$ is to the left of $T$, otherwise $J$ is to the right of $T$;
	\item $x_{q_b}<0$;
	\item if $s=+$ then $x_{q_{b+1}}>0$, otherwise $x_{q_{b+1}}<0$;
	\item if $s=+$ then the orientation of $f|_J$ coincides with the orientation of the closest to it branch of $f|_T$, otherwise the orientation of $f|_J$ is opposite to the orientation of the closest to it branch of $f|_T$;
	\item signs of $x_{q_n}$ change with with interval $2$ (exactly as for $\theta$-recurrent unimodal maps, but two first points of closest recurrence can have the same sign due to item (2));
	\item relations (2),(3) of Theorem~\ref{lmm:signs} are satisfied;
	\item   
	\begin{enumerate}
		\item if $a_1>2$ or $a_1=1,a_2>1$, i.e., $q_{b+1}>2$ and $J$ is to the right of $T$, then $x_1<0<-x_1<x_2<...<x_{q_{b+1}-2}$ and $x_{q_{b+1}}$ are all in $T$ while $x_{q_{b+1}-1}\in J$,
		\item if $a_1=2$ or $a_1=1,a_2=1$, i.e., $q_{b+1}=2$ and $J$ is to the left of $T$, then $x_b\in J$, while $x_{q_{b+1}}\in T$;
	\end{enumerate} 
	\item if $\varkappa=\sum_{i=0}^\infty\gamma_i q_i$ with $\gamma_k$ first non-zero term, then $x_\varkappa$ and $x_{\gamma_k q_k}$ together belong either to $T$ or $J$ (it follows from the previous item that to $J$ belong only either all $x_\varkappa$ with $\varkappa$ having $\gamma_b=q_{b+1}-1$, or all $x_\varkappa$ with $\varkappa$ having $\gamma_b=1$);
\end{enumerate}

The existence of such functions $f\in\mathcal{A}_\theta^s$ for every pair $\theta,s$ is not difficult to show by construction as in Theorem~\ref{thm:map_construction}. However, we do not need it: a function from $\mathcal{A}_{[1,1,1,a_4,...]}^{+}$ can be obtained from a $\theta$-recurrent unimodal map by a simple surgery, so we will only define a renormalization operator
$$\mathcal{R}:\bigcup_{\theta,s}\mathcal{A}_\theta^s\to\bigcup_{\theta,s}\mathcal{A}_\theta^s.$$

On the set of irrational angles $\theta=[a_1,a_2,a_3,...]$ define the shift map $\sigma$ so that $\sigma([b_1,b_2,b_3,...])$ is equal either to $[b_1-1,b_2,b_3,...]$ if $b_1>1$, or to $[b_2,b_3,b_4,...]$ otherwise. Also, let $-s$ denote the sign opposite to $s$.

Here is the formal definition of $\mathcal{R}$.
\begin{itemize}
	\item ($q_{b+1}>2$) For $\varkappa=\sum_{i=b}^\infty\gamma_i q_i=[\gamma_b\gamma_{b+1}\gamma_{b+2}\cdots]$, $x_\varkappa\in J$ iff $\gamma_{b}=q_{b+1}-1>1$. Let $J_1\subset T$ be the maximal closed interval containing $x_{q_{b+1}-2}$, not containing $0$ and such that $f^2(J_1)=T$, and let $T_1:=(f|_T)^{-1}(T)$. After rescaling of the map $g:J_1\cup T_1\to I_1=T$ such that $g|_{J_1}=f^2,g|_{T_1}=f$ we obtain the map $\mathcal{R}f\in\mathcal{A}_{\sigma(\theta)}^s$.
	Indeed, indices of postcritical set change as follows: for $\varkappa=\sum_{i=b}^\infty\gamma_i q_i$ with $\gamma_{b}=q_{b+1}-1$, i.e., for $x_\varkappa\in J$, the points $x_\varkappa$ disappear and for the rest of indices $\varkappa$ their representation $[\gamma_b\gamma_{b+1}\gamma_{b+2}\cdots]$ stays the same though corresponds to the Ostrowski numeration system associated to $\sigma(\theta)$. Note also that the orientation of $\mathcal{R}f|_{J_1}$ coincides with the orientation of $f|_J$ if $q_{b+1}>3$, and changes to the opposite if $q_{b+1}=3$. 
	\item ($q_{b+1}=2$) For $\varkappa=\sum_{i=b}^\infty\gamma_i q_i=[\gamma_b\gamma_{b+1}\gamma_{b+2}\cdots]$, $x_\varkappa\in J$ iff $\gamma_{b}=q_b=1$. Let $T_1\subset T$ be the maximal closed interval containing $f^2(T_1)\subset T$, and let $J_1$ be the maximal closed interval containing $x_{q_{b+2}-1}=x_{a_{b+2}q_{b+1}}$ such that $f(J_1)\subset T$. After rescaling of the map $g:J_1\cup T_1\to I_1=T$ such that $g|_{J_1}=f,g|_{T_1}=f^2$ we obtain the map $\mathcal{R}f\in\mathcal{A}_{\sigma(\theta)}^{-s}$.
	In fact, indices of postcritical set change as follows: for $\varkappa=\sum_{i=b}^\infty\gamma_i q_i$ with $\gamma_{b}=q_b=1$, i.e., for $x_\varkappa\in J$, the points $x_\varkappa$ disappear and for the rest of indices $\varkappa$ their representation $[\gamma_b\gamma_{b+1}\gamma_{b+2\cdots}]$ changes to the shifted presentation $[\gamma_{b+1}\gamma_{b+2}\cdots]$ in the Ostrowski numeration system associated to $\sigma(\theta)$. If $q_{b+2}>3$, then $\mathcal{R}f|_{J_1}$ has the opposite orientation to that of the branch of $f|_T$ containing $x_{q_{b+1}}$, whence stays the same if $s=-$ and changes otherwise. If $q_{b+2}=3$, then $\mathcal{R}f|_{J_1}$ has the same orientation to that of the branch of $f|_T$ containing $x_{q_{b+1}}$, whence stays the same if $s=+$ and changes otherwise.
\end{itemize}

To sum up, the operator $\mathcal{R}$ will act in the following way:
\begin{itemize}
	\item if $q_{b+1}>2$, then $\mathcal{R}(\mathcal{A}_\theta^s)\subset\mathcal{A}_{\sigma(\theta)}^s$;
	\item if $q_{b+1}=2$, then $\mathcal{R}(\mathcal{A}_\theta^s)\subset\mathcal{A}_{\sigma(\theta)}^{-s}$. 
\end{itemize}

An example of such renormalization is presented on Picture~\ref{pic:renorm}.

Finally, we are ready to finish the proof of Theorem~\ref{thm:hausdorff_dim_bounded}, that is we prove the next lemma.

\begin{lmm}
	If $\theta$ is of bounded type, then every $\theta$-recurrent  $C^2$ map with non-flat critical point belongs to $\mathcal{F}_0(\theta)$.
\end{lmm}
\begin{proof}
	The proof goes exactly as in case of Fibonacci maps in \cite{LM}. So, we only sketch the scheme of the proof and give references. Note that in the notation of \cite{LM} $J^n$ corresponds to our $J^n_{q_{n-1}}$.
	
	First, we do a surgery of the map $f$ and obtain a map $\tilde{f}\in\mathcal{A}$ as in \cite[Paragraph after Lemma~6.4]{LM}. The sequence of $\tilde{f}$. Since $\theta$ is of bounded type, the sequence $\mathcal{R}_n\tilde{f}$ has a subsequence converging to some $g\in\mathcal{E}\subset\mathcal{A}$ where $\mathcal{E}$ is a certain space of analytic maps (for precise definitions of $\mathcal{E}$ and topology see \cite[Section 6]{LM}). One may assume that $g\in\mathcal{A}^{-}$ in the notation of \cite{LM}, otherwise do a few more renormalizations.
	
	Again, after additional renormalization this map $g$ can be made into a polynomial-like map $h$ of type (2,1) (see \cite[Section~8]{LM}). On the other hand, all such maps are quasi-symmetrically conjugate (\cite[Corollary~7.4]{LM}). But exactly as in \cite[Example~7.1]{LM} one can construct a polynomial-like map of type (2,1) with an arbitrarily small $\delta_1^1$. Hence $f\in\mathcal{F}_0(\theta)$.
\end{proof}

\section{$\theta$ with slow growth of denominators}

Some of our estimates can be generalized to the case of angle $\theta$ with sufficiently slow growth of denominators if we restrict to quadratic polynomials. However, we must begin in a slightly more general setting.

Let $f\in\mathcal{A}$ have a non-positive Schwarzian derivative and be equal to the quadratic polynomial $x^2+c$ near $0$ (note that analogous \cite[Example~7.1]{LM} with the corresponding $\theta$ is an example of such map).

First, we prove a few more elaborate estimates on $\delta_n^i$'s. For now, no bounds on $\theta$ are considered.

Note, that from the inequality
$$\frac{1}{1-y^2}\leq\left(\frac{1+x}{1-x}\right)^2$$
for positive $x,y$ follows that
$$y<2\sqrt{x}.$$
Since in our case the $O(.)$ bounds of Lemma~\ref{lmm:a_priori_2} are equal to $0$, we obtain for $i<a_{n+1}$
$$\delta_n^i<2\sqrt{\delta_n^{i+1}},$$
and
$$\delta_n^{a_{n+1}}<2\sqrt{\delta_{n-1}^1}.$$
It follows immediately that $\alpha_{n+1}<4(\alpha_n)^{1/2^{a_{n+2}}}$.

Now, we can obtain a more elaborate version of Lemma~\ref{lmm:derivatives} in the new setting. Recall that according to our notation $\alpha_n=\delta_{n+1}^{a_{n+2}}$.

\begin{lmm}
	\label{lmm:derivatives2}
	There is a constant $C>4$, not depending on $f$ and $\theta$ such that the following statements take place. Denote $K_n:=1+C\alpha_n$.
	
	For every $x\in[x_1,x_{1+a_{n+2}q_{n+1}}]$,
	$$\frac{1}{K_{n-1}}\frac{d_n^1}{(d_{n+1}^{a_{n+2}})^2}<\absv{(f^{q_n-1})'(x)}=K_{n-1}\left(1+\alpha_{n+1}\alpha_n\right)\frac{d_n^1}{(d_{n+1}^{a_{n+2}})^2}.$$
	
	If $a_{n+1}>1$, then for $0<i<a_{n+1}$ and $x=x(i)\in [x_1,x_{1+iq_n}]$,
	$$\frac{1}{K_{n-1}}\frac{d_n^{i+1}}{(d_n^i)^2}<\absv{(f^{q_n-1})'(x)}<K_{n-1}\left(1+\delta_n^1\delta_n^2\cdots\delta_n^i\right)\frac{d_n^{i+1}}{(d_n^i)^2}.$$
	
\end{lmm}
\begin{proof}
	
	Let $x\in[x_1,x_{1+a_{n+2}q_{n+1}}]$. Consider the homeomorphism
	$$f^{q_n-1}:\left(H_1^n,[x_1,x_{1+a_{n+2}q_{n+1}}]\right)\to \left(T^{n-2},[x_{q_n},x_{q_{n+2}}]\right).$$
	From Koebe principle (for $f^{-1}$) and bounds in Lemma~\ref{lmm:area_bounds} with $O(.)=0$ for any $\xi\in[x,x_{1+a_{n+2}q_{n+1}}]$ we have
	$$\frac{(f^{q_n-1})'(x)}{(f^{q_n-1})'(\xi)}=1+O\left(\left[T^n:T^{n-2}\right]\right)<1+C\lambda_n\lambda_{n-1}<1+C\alpha_{n-1},$$
	where the constant $C$ does not depend neither on $\theta$, nor on $f$. By the Mean Value Theorem, there exists $\xi$ such that
	$$\absv{(f^{q_n-1})'(\xi)}=\frac{d_n^1+d_{n+2}^1}{(d_{n+1}^{a_{n+2}})^2}=\frac{d_n^1}{(d_{n+1}^{a_{n+2}})^2}\left(1+\lambda_{n+2}\lambda_{n+1}\right)<\frac{d_n^1}{(d_{n+1}^{a_{n+2}})^2}\left(1+\alpha_{n+1}\alpha_n\right).$$
	Thus, we obtain the inequality
	$$\frac{1}{K_{n-1}}\frac{d_n^1}{(d_{n+1}^{a_{n+2}})^2}<\absv{(f^{q_n-1})'(x)}=K_{n-1}\left(1+\alpha_{n+1}\alpha_n\right)\frac{d_n^1}{(d_{n+1}^{a_{n+2}})^2}.$$
	
	Next, if $a_{n+1}>1$, for $0<i<a_{n+1}$ consider the homeomorphism
	$$f^{q_n-1}:\left(H_1^n,[x_1,x_{1+iq_n}]\right)\to \left(T^{n-2},[x_{q_n},x_{(i+1)q_n}]\right).$$
	Repeating the argument above we get
	$$\frac{(f^{q_n-1})'(x)}{(f^{q_n-1})'(\xi)}=1+O\left(\left[[x_{q_n},x_{(i+1)q_n}]:T^{n-2}\right]+q^n\right)<K_{n-1}$$
	and
	$$\absv{(f^{q_n-1})'(\xi)}=\frac{d_n^1+d_n^{i+1}}{(d_n^i)^2}=\frac{d_n^{i+1}}{(d_n^i)^2}\left(1+\delta_n^1\delta_n^2\cdots\delta_n^i\right).$$
	Hence,
	$$\frac{1}{K_{n-1}}\frac{d_n^{i+1}}{(d_n^i)^2}<\absv{(f^{q_n-1})'(x)}<K_{n-1}\left(1+\delta_n^1\delta_n^2\cdots\delta_n^i\right)\frac{d_n^{i+1}}{(d_n^i)^2}.$$
\end{proof}

The next lemma tell how the quantity $\alpha_n^2\lambda_n$ changes with $n$ sand estimates the ratio of $\delta_n^k$ and $\alpha_n^{2^k}$.

\begin{lmm}
	\label{lmm:K_bounds}
	Let $C>4$ be the constant from Lemma~\ref{lmm:derivatives2}. Then
	$$\alpha_{n+1}^2 \lambda_{n+1}<\frac{K_{n+1}K_n K_{n-1}^{a_{n+1}} K_{n-2}}{2^{a_{n+1}}}\alpha_n^2\lambda_n.$$
	
	If, additionally, $\alpha_n<1/2, K_{n-1}<\sqrt{2}$ and $K_n<2$, then for $0<k<a_{n+1}$,
	$$\frac{1}{(K_{n-1}^2 K_n)^{{2^k}-1}}<\frac{\delta_n^k}{\alpha_n^{2^k}}<(K_{n-1}^2 K_n)^{{2^k}-1}.$$
\end{lmm}
\begin{proof}
		Substitute the estimates from Lemma~\ref{lmm:derivatives2} into expression
	$$(f^{q_{n+1}-1})'(x_1)=(f^{q_n-1})'(x_1)\cdot2x_{q_n}\cdot(f^{q_n-1})'(x_{1+q_n})\cdot2x_{2q_n}\times\cdots$$
	$$\times(f^{q_n-1})'(x_{1+(a_{n+1}-1)q_n})\cdot2x_{a_{n+1}q_n}\cdot(f^{q_{n-1}-1})'(x_{1+a_{n+1}q_n})$$
	and obtain	
	$$K_n\left(1+\alpha_{n+2}\alpha_{n+1}\right)\frac{d_{n+1}^1}{(d_{n+2}^{a_{n+3}})^2}>\frac{1}{K_{n-1}}\frac{d_n^1}{(d_{n+1}^{a_{n+2}})^2}\cdot2d_n^1\cdot\frac{1}{K_{n-1}}\frac{d_n^{2}}{(d_n^1)^2}\cdot2d_n^2\times\cdots$$
	$$\times\frac{1}{K_{n-1}}\frac{d_n^{a_{n+1}}}{(d_n^{a_{n+1}-1})^2}\cdot2d_n^{a_{n+1}}\cdot\frac{1}{K_{n-2}}\frac{d_{n-1}^1}{(d_n^{a_{n+1}})^2}.$$
	
	Thus,
	$$\alpha_{n+1}^2 \lambda_{n+1}<\frac{K_n K_{n-1}^{a_{n+1}} K_{n-2}(1+\alpha_{n+2}\alpha_{n+1})}{2^{a_{n+1}}}\alpha_n^2\lambda_n<\frac{K_{n+1}K_n K_{n-1}^{a_{n+1}} K_{n-2}}{2^{a_{n+1}}}\alpha_n^2\lambda_n.$$
	
	From the estimates of the previous lemma for $x=x_1$ we have 
	$$\frac{1}{K_{n-1}}\frac{d_n^1}{(d_{n+1}^{a_{n+2}})^2}<K_{n-1}\left(1+\delta_n^1\right)\frac{d_n^2}{(d_n^1)^2},$$
	or,
	$$\delta_n^1<\alpha_n^2 K_{n-1}^2\left(1+\delta_n^1\right).$$
	Hence,
	$$\delta_n^1<\frac{\alpha_n^2 K_{n-1}^2}{1-\alpha_n^2 K_{n-1}^2}<\alpha_n^2 K_{n-1}^2(1+2\alpha_n^2 K_{n-1}^2)<\alpha_n^2 K_{n-1}^2(1+C\alpha_n)=\alpha_n^2 K_{n-1}^2 K_n.$$
	
	On the other hand, from the inequality
	$$\frac{1}{K_{n-1}}\frac{d_n^{2}}{(d_n^1)^2}<K_{n-1}\left(1+\alpha_{n+1}\alpha_n\right)\frac{d_n^1}{(d_{n+1}^{a_{n+2}})^2}$$
	we obtain
	$$\delta_n^1>\frac{1}{K_{n-1}^2 K_n}\alpha_n^2.$$
	
	Further, we also have 
	$$\frac{1}{K_{n-1}}\frac{d_n^{i+1}}{(d_n^i)^2}<K_{n-1}\left(1+\delta_n^1\delta_n^2\cdots\delta_n^{i+1}\right)\frac{d_n^{i+2}}{(d_n^{i+1})^2},$$
	whence
	$$\delta_n^{i+1}<K_{n-1}^2\left(1+\delta_n^1\delta_n^2\cdots\delta_n^{i+1}\right)(\delta_n^i)^2<K_{n-1}^2\left(1+\alpha_n^2 K_{n-1}^2 K_n\right)(\delta_n^i)^2<K_{n-1}^2 K_n(\delta_n^i)^2.$$
	
	Considering the corresponding ``reverse'' inequality
	$$\frac{1}{K_{n-1}}\frac{d_n^{i+2}}{(d_n^{i+1})^2}<K_{n-1}\left(1+\delta_n^1\delta_n^2\cdots\delta_n^i\right)\frac{d_n^{i+1}}{(d_n^i)^2},$$
	we get
	$$\delta_n^{i+1}>\frac{1}{K_{n-1}^2 K_n} (\delta_n^i)^2.$$
	
	Multiplying all inequalities from $1$ to $k$ one obtains
	$$\delta_n^k<(K_{n-1}^2 K_n)^{1+2+2^2+\dots+2^{k-1}}\alpha_n^{2^k}<(K_{n-1}^2 K_n)^{{2^k}-1}\alpha_n^{2^k},$$
	and analogously for the lower bound.
\end{proof}

Using previous computation we provide the estimate on how $\alpha_n^{2^{a_{n+1}}}\alpha_{n-1}$ changes when $n$ increases.

\begin{lmm}
	\label{lmm:decrease}
	If $\alpha_n,\alpha_{n+1}<1/2$, $K_{n-1},K_n,K_{n+1}<\sqrt{2}$, $a_{n+1},a_{n+2}\leq A$ and $M=\max\{K_{n+1},K_n,K_{n-1},K_{n-2}\}$, then
	$$\frac{\alpha_{n+1}^{2^{a_{n+2}}}\alpha_n}{\alpha_n^{2^{a_{n+1}}}\alpha_{n-1}}<\frac{M^{2^{A+3}}}{2^{a_{n+1}}}.$$
\end{lmm}
\begin{proof}
	We simply make use of Lemma~\ref{lmm:K_bounds}. Since $\lambda_n=\delta_n^1\delta_n^2\cdots\dots\cdot\delta_n^{a_{n+1}}$,
	$$\alpha_{n+1}^{2^{a_{n+2}}}\alpha_n \frac{1}{(K_n^2 K_{n+1})^{2^{a_{n+2}}-a_{n+2}}}<\frac{K_{n+1}K_n K_{n-1}^{a_{n+1}} K_{n-2}}{2^{a_{n+1}}}\alpha_n^{2^{a_{n+1}}}\alpha_{n-1}(K_{n-1}^2 K_n)^{2^{a_{n+1}}-a_{n+1}}.$$
	Hence,
	$$\frac{\alpha_{n+1}^{2^{a_{n+2}}}\alpha_n}{\alpha_n^{2^{a_{n+1}}}\alpha_{n-1}}<\frac{K_{n+1}^{2^{a_{n+2}}-a_{n+2}+1}K_n^{2(2^{a_{n+2}}-a_{n+2})+2^{a_{n+1}}-a_{n+1}+1}K_{n-1}^{2(2^{a_{n+1}}-a_{n+1})+a_{n+1}}K_{n-2}}{2^{a_{n+1}}}<$$
	$$\frac{M^{6\cdot2^A}}{2^{a_{n+1}}}<\frac{M^{2^{A+3}}}{2^{a_{n+1}}}.$$
	
\end{proof}

The next statement tells precisely how small must be $\alpha_{n-2}$ so $\alpha_m$ decrease exponentially subject to the condition that $a_n\geq\{a_{n-1},a_n,\dots,a_m\}$ and gives an estimate on the speed of decrease.

\begin{lmm}
	\label{lmm:big_lambdas_decrease}
	Assume that $a_{n-1},a_n,\dots,a_m\leq A=a_{n+1}$ for some $m>n$. If $\alpha_{n-2}<2^{-(1+\varDelta)N2^{4N}}$ for some $\varDelta>0$ and $A=A(\epsilon)$ is big enough, then for $0<k<m-n$ and $0<\epsilon<1$,
	$$\alpha_{n+k}<\left(\frac{1+\epsilon}{2}\right)^{(k-1)/A2^A}2^{-(1+\varDelta/2)A2^A}.$$
\end{lmm}
\begin{proof}
	Denote $M_i=\max\{K_{i+1},K_i,K_{i-1},K_{i-2}\}$ and $\Lambda_n=\alpha_n^{2^{a_{n+1}}}\alpha_{n-1}$.
	
	Let $x=\alpha_{n-2}$. Then $\alpha_{n+i}<4(x)^{1/2^{iA}}$, for $i\leq m+2$. From Lemma~\ref{lmm:decrease} we know that if $x$ is sufficiently small, then
	\begin{equation}
		\label{eqn:big_lambdas}
		\Lambda_{n+1}<\frac{M_n^{2^{A+3}}}{2^A}\Lambda_n<\frac{M_n^{2^{A+3}}}{2}\Lambda_n.
	\end{equation}
	We want to show that $M_n^{2^{A+3}}<1+\epsilon$ for $A$ big enough. It is enough to prove the inequality
	$$\left(1+4Cx^{(1/2)^{3A}}\right)^{2^{A+3}}<\left(1+4Cx^{(1/2)^{4A}}\right)^{2^{A+3}}<1+\epsilon.$$
	We want to ``replace'' $3A$ by $4A$ here to use the inequality again during the next steps of induction. We have
	$$\frac{(1+\epsilon)^{1/2^{A+3}}-1}{4C}>\frac{\epsilon2^{-A}}{32C}>2^{-(1+\varDelta)A}>x^{(1/2)^{4A}}.$$
	
	Thus, from \ref{eqn:big_lambdas} follows that
	$$\alpha_{n+1}<\left(\frac{1+\epsilon}{2} \right)^{1/A}\max\{\alpha_n,\alpha_{n-1}\}<\left(\frac{1+\epsilon}{2} \right)^{1/A}\cdot2^{-(1+\varDelta/2)A2^{2A}}.$$
	
	Repeating the discussion above for $\alpha_{n+2},\alpha_{n+1},\alpha_n,\alpha_{n-1}$ with the estimate for $M_{n+1}$ not bigger than for $M_n$, we get
	$$\Lambda_{n+2}<\frac{M_{n+1}^{2^{A+3}}M_n^{2^{A+3}}}{2^2}\Lambda_n<\left(\frac{1+\epsilon}{2}\right)^2\Lambda_n.$$
	Hence,
	$$\min\{\alpha_{n+2},\alpha_{n+1}\}<\left(\frac{1+\epsilon}{2}\right)^{2/A}\max\{\alpha_n,\alpha_{n-1}\}.$$
	
	If $\alpha_{n+2}<\alpha_{n+1}$, we get the corresponding estimate for $\alpha_{n+2}$. Otherwise we have $$\alpha_{n+2}<4(\alpha_{n+1})^{1/2^A}<4\cdot\left(\frac{1+\epsilon}{2}\right)^{1/A2^A}4\cdot2^{-(1+\varDelta)A2^A}<\left(\frac{1+\epsilon}{2}\right)^{1/A2^A}2^{-(1+\varDelta/2)A2^A}.$$
	
	Repeating the same procedure for $\alpha_{n+3},\alpha_{n+2},\alpha_{n+1},\alpha_n$ we obtain 
	$$\min\{\alpha_{n+3},\alpha_{n+2}\}<\left(\frac{1+\epsilon}{2}\right)^{3/A}\max\{\alpha_n,\alpha_{n-1}\}.$$
	
	Note that either $\alpha_{n+2}<\alpha_{n+3}$ and we improve the estimate for $\alpha_{n+2}$:
	$$\alpha_{n+2}<\left(\frac{1+\epsilon}{2} \right)^{3/A}\cdot2^{-(1+\varDelta/2)A2^{2A}},$$
	or we get the same estimate for $\alpha_{n+3}$.
	
	It is easy to see that this dichotomy preserves on the next steps as well. The proof follows.	
\end{proof}

Finally, we are ready to prove Theorem~\ref{thm:hausdorff_dim_unbounded}.

\begin{proof}[Proof of Theorem~\ref{thm:hausdorff_dim_unbounded}]
	As for the Fibonacci map, quadratic $f(x)=x^2+c$ can be renormalized to a polynomial-like map of type (2,1) (see \cite[Section~7. Renormalization of a quadratic-like Fibonacci map]{LM}). All such maps are quasi-symmetrically conjugate by \cite[Corollary~7.4]{LM}. Hence, it is enough to prove the theorem for an arbitrary representative. The construction in \cite[Example~7.1]{LM} works for arbitrary $\theta$, so we may assume that the first finitely many $\delta_n^i$'s are as small as needed after choosing an appropriate representative of the conjugacy class.
	
	In particular, we may assume that for some $i$, $\alpha_{{N_i}-2}$ satisfies conditions of Lemma~\ref{lmm:big_lambdas_decrease} for $A=a_{N_{i}}$ and some $\delta>0$ which we choose later. We have
	$$\alpha_{N_i+k}<\left(\frac{1+\epsilon}{2}\right)^{(k-1)/A2^A}2^{-(1+\varDelta/2)A2^A},$$
	and want to compute how big should be $N_{i+1}$ (depending on $a_{N_{i+1}}$) so that $\alpha_{N_{i+1}-2}$, obtained from the estimate above, satisfied the conditions of Lemma~\ref{lmm:big_lambdas_decrease} as well, but for $A=N_{i+1}$. Thus, we have inequality
	$$\left(\frac{1+\epsilon}{2}\right)^{(N_{i+1}-N_i-3)/a_{N_i} 2^{a_{N_i}}}2^{-(1+\varDelta/2)a_{N_i} 2^{a_{N_i}}}<2^{-(1+\varDelta)a_{N_{i+1}}2^{4a_{N_{i+1}}}}.$$
	This is equivalent to
	$$N_{i+1}-N_i>3+\frac{\log 2}{\log\left(\frac{2}{1+\epsilon}\right)}\left((1+\varDelta)a_{N_{i+1}}2^{4a_{N_{i+1}}}-(1+\varDelta/2)a_{N_i} 2^{a_{N_i}}\right)a_{N_i} 2^{a_{N_i}}.$$
	The right hand side smaller than
	$$3+\frac{\log 2}{\log\left(\frac{2}{1+\epsilon}\right)}(1+\varDelta)a_{N_i}a_{N_{i+1}}2^{4a_{N_{i+1}}+a_{N_i}}<2^{(5+\tau)a_{N_{i+1}}}$$
	if $N_i$ is big enough. Hence, if $N_i$ satisfies conditions of Theorem~\ref{thm:hausdorff_dim_unbounded}, we have bounds on $\alpha_{N_i+k}$ for $0<k\leq N_{i+1}-N_i$:
	$$\alpha_{N_i+k}<\left(\frac{1+\epsilon}{2}\right)^{(k-1)/a_{N_i} 2^{a_{N_i}}}2^{-(1+\varDelta/2)a_{N_i} 2^{a_{N_i}}}.$$
	
	Denote $\gamma_i=\left(\frac{1+\epsilon}{2}\right)^{1/a_{N_i} 2^{a_{N_i}}}$ and $D_n=d_n^{a_{n+1}}$. We obtain
	$$D_{N_i+k}<D_{N_i}\gamma_i^{1+2+\dots+(k-1)}=D_{N_i}\gamma_i^{(k^2-k)/2}$$
	for $0<k\leq N_{i+1}-N_i$.
	
	When $k$ is close to $N_{i+1}-N_i$,
	$$D_{N_i+k}<D_{N_i} \gamma_i^{N_{i+1}^2/3}<D_{N_i}\left(\frac{1+\epsilon}{2}\right)^{\frac{N_{i+1}^2}{3a_{N_i} 2^{a_{N_i}}}}<D_{N_i}\left(\frac{1+\epsilon}{2}\right)^{N_{i+1}^{3/2}}.$$
	
	At the same time, the amount of intervals in $M^{N_{i}+k}$ is less than $$q_{N_i}(a_{N_{i}}+1)^{k}<q_{N_i}e^{2N_{i+1}\log(a_{N_{i+1}})}.$$
	
	As in the Lemma~\ref{lmm:dim=0}, one obtains that all intervals of $M^n$ are of comparable size for big $n$. That is, the lengths of these intervals decrease much faster than their number whatever Hausdorff measure we choose (exactly as in Lemma~\ref{lmm:dim=0}). Hence, the Hausdorff  dimension of $\mathcal{O}$ is equal to $0$ in this case as well.
\end{proof}

\section{Appendix}
\label{sec:appendix}

\subsection{Irrational rotations and Ostrowski numeration system}

Details and proofs can be found in \cite{Arnoux}.

Every irrational angle $\theta\in (0,1)$ is uniquely represented by its continued fraction
$$\theta=[a_1,a_2,a_3,...]=\cfrac{1}{a_1+\cfrac{1}{a_2+\cfrac{1}{a_3+\cdots}}}, a_n\in\mathbb{N},$$
which encodes dynamical properties of the circle rotation by $\theta$. In particular, the denominators $q_n$ of the truncated fractions
$$\frac{p_n}{q_n}=[a_1,a_2,...,a_n]=\cfrac{1}{a_1+\cfrac{1}{\ddots+\cfrac{1}{a_n}}}$$ are exactly the times of closest recurrence of $0$ under rotation by $\theta$, except that in the case $a_1>1$ one more time of closest recurrence is $q_0=1$. They can be computed by the recurrent formula $q_{n+1}=a_{n+1} q_n+q_{n-1}$. Analogous formula holds for $p_n$: $p_{n+1}=a_{n+1} p_n+p_{n-1}$ if we assume $p_0=0$.

Given an irrational $\theta=[a_1,a_2,...]\in(0,1)$, every integer $k$ can be written in a unique way as a sum $k=\sum_{n=0}^{\infty}\gamma_n q_n$ where $0\leq\gamma_0\leq a_1-1$ and $0\leq\gamma_n\leq a_{n+1}$ for $n>0$, only finitely many $\gamma_n$'s are nonzero and if the digit $\gamma_n$ is equal to $a_{n+1}$, then $\gamma_{n-1}=0$. One represents such $k$ as $[\gamma_0\gamma_1\gamma_2\cdots]$ starting from the smallest term. The representation $[00...00\gamma_n 00\cdots]$ with $\gamma_n=1$ corresponds to $k=q_n$. Further, one can also consider formal infinite sums $\varkappa=\sum_{n=1}^{\infty}\gamma_n q_n$ where $0\leq\gamma_0\leq a_1-1$ and $0\leq\gamma_n\leq a_{n+1}$ for $n>0$ and if the digit $\gamma_n$ is equal to $a_{n+1}$, then $\gamma_{n-1}=0$. Such infinite sums are limits of finite words $[\gamma_0\gamma_1...\gamma_m00\cdots]$ in the product topology on $\mathbb{N}_0^{\mathbb{N}_0}$.

Furthermore, one can obtain a similar representation of all real numbers on the unit circle. Denote $\theta_n:=q_n\alpha-p_n$. Then every real $x\in[-\theta,1-\theta)$ can be presented (non-uniquely) as the infinite sum
\begin{equation}
	\label{eqn:ostrowski_reals}
	x=\sum_{n=0}^\infty \gamma_n\theta_n,
\end{equation}
where $0\leq\gamma_0\leq a_1-1$ and $0\leq\gamma_n\leq a_{n+1}$ for $n>0$ and if the digit $\gamma_n$ is equal to $a_{n+1}$, then $\gamma_{n-1}=0$. If we assume additionally $\gamma_n\neq a_{n+1}$ for infinitely many even integers $n$, then the representation in formula~\ref{eqn:ostrowski_reals} is unique. In this setting addition of $1$ to $[\gamma_0\gamma_1\gamma_2\cdots]$ corresponds to rotation of $x$ either by $\theta$ if $a_1>1$, or by $-\theta$ otherwise.

Let $R=R(\theta)$ be a rotation operator of the unit circle $S^1$ by angle $\theta\in(0,1)$. The \emph{rotation sequence} $s(x)=(s_0(x)s_1(x)s_2(x)...)$ of a point $x\in S^1$ is defined as follows. For $n\geq 0$, $s_n(x)=0$, if $R^n(x)\in (-\theta,0]$, then $s_n(x)=1$, and $s_n(x)=0$ otherwise (for full generality one would have to consider additionally the definition with interval $[-\theta,0)$ but in our setting it is not necessary). We consider only $s(\theta)$ for irrational $\theta$. In this case $s(\theta)$ is the so called (``left special'') Sturmian word. A basic example is the Fibonacci word $s((\sqrt{5}-1)/2)=s([1,1,1,1...])$.

From the definitions it is easy to see that for every Sturmian word either ``1'' or ``0'' is isolated, that is, does not appear twice in a row; moreover, $s(\theta)$ starts with the symbol which is not isolated. Having this in mind, one can define a ``recoding'' (or ``compression'') of the Sturmian word: if ``0'' is isolated, then we replace every neighboring pair of symbols ``10'' by ``1''; if ``1'' is isolated, then we replace every neighboring pair of symbols ``01'' by ``0''. One can show that the recoded sequence is again a Sturmian word. Moreover, if $\theta=[a_1,a_2,a_3,...]$, the recoded word coincides with $s(\sigma(\theta))$ where $\sigma([b_1,b_2,b_3,...])$ is by definition equal either to $[b_1-1,b_2,b_3,...]$ if $b_1>1$, or to $[b_2,b_3,b_4,...]$ otherwise. This recoding encodes symbolically a renormalization of a circle rotation and can be iterated infinitely many times. 

\subsection{Schwarz lemma and Koebe principle}

We use the same statements of Schwarz lemma and Koebe principle as in the appendix of \cite{LM}. For convenience of the reader we provide the appendix here almost without changes.

Consider four points $a<b<c<d$ and two nested intervals $L=[a,d]$ and $H=[b,c]$. The \emph{Poincar\'e length} of $H$ in $L$ is defined as
$$[H:L]:=\log\frac{(d-b)(c-a)}{(d-c)(b-a)}.$$

For a $C^3$ diffeomorphism $g:(L,H)\to(L',H')$ its Schwarzian derivative is defined as 
$$Sg:=\frac{g'''}{g'}-\frac{3}{2}\left(\frac{g''}{g'}\right)^2.$$

Note that for a quadratic polynomial $x^2+c$ its Schwarzian derivative $-3/2x^2$ is strictly negative away from the critical point.

\begin{lmm}[Schwarz Lemma]
	If $f$ has non-negative Schwarzian derivative, then it contracts Poincar\'e length $[H':L']\leq[H:L]$.
\end{lmm}

\begin{lmm}[Koebe principle]
	Let $g$ has non-negative Schwarzian derivative. If $[H:L]\leq l$, then $\absv{g'(x)/g'(y)}\leq K(l)$ for any $x,y\in H$ and $K(l)=1+O(l)$ as $l\to 0$.
\end{lmm}

These two lemmas can also be generalized. Consider a chain of interval diffeomorphisms
$$I_1\to J_1\to\cdots\to I_n\to J_n$$
where $g_i:I_i\to J_i$ have non-negative Schwarzian derivative while  $h_i:J_i\to I_{i+1}$ are $C^2$ smooth. Denote $F:=g_n\circ h_{n-1}\circ g_{n-1}\circ\cdots\circ h_1\circ g_1$, and let $G_i\subset\interior I_i$ and $H_i\subset\interior J_i$ be closed subintervals related by diffeomorphisms.

Denote by $\mathbf{h}$ the family of maps $h_i$, by $\mathbf{I}$ the family of intervals $I_i$, etc. Let $\abs{h_i}=\max\absv{h_i''(x)/h_i'(x)}$, $\abs{\mathbf{h}}=\max\abs{h_i}$ be the ``maximal non-linearity'' of $\mathbf{h}$, $\absv{\mathbf{I}}=\sum\absv{I_i}$ be the total length of $\mathbf{I}$, $l=[G_1:I_1]$.

\begin{lmm}[Schwarz Lemma, smooth version]
	Expansion of the Poincar\'e length by the map $F$ is controlled by $\mathbf{h}$ in the manner
	$$[H_n:J_n]\leq l+O(\absv{\mathbf{J}})$$
	with the constant depending on $\abs{\mathbf{h}}$.
\end{lmm}

\begin{lmm}[Koebe principle, smooth version]
	Distortion of $F|_{G_1}$ can be estimated as
	$$\left|\frac{F'(x)}{F'(y)}\right|\leq K(l;\abs{\mathbf{h}},\absv{\mathbf{J}})$$
	where $K=1+O(l+\absv{\mathbf{J}})$ as $\absv{\mathbf{J}},l\to 0$ with the constant depending on $\abs{\mathbf{h}}$.
\end{lmm}

\begin{sidewaysfigure}[h]
	\vspace{5cm}
	\includegraphics[width=\textwidth]{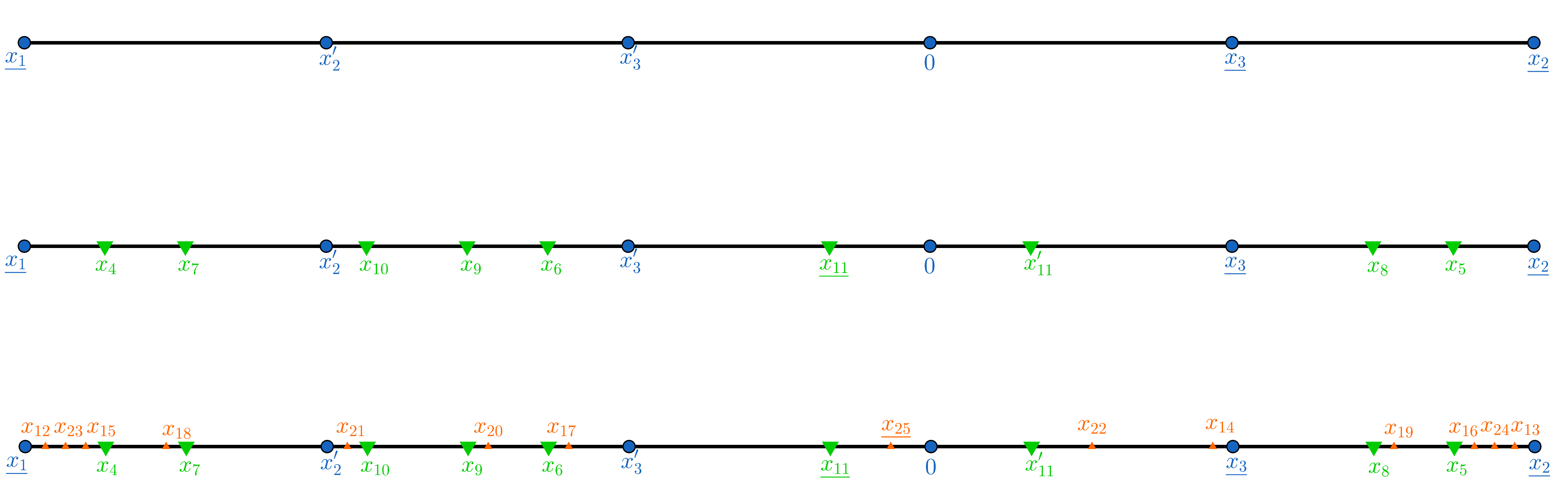}
	\caption{Schematic picture of critical orbit combinatorics of a $\theta$-recurrent map $f$ for $\theta=[1,1,1,3,2,...]$. Closest recurrence points are $\underline{x_1},\underline{x_2},\underline{x_3},\underline{x_{11}},\underline{x_{25}},\dots$. By $x_2,x_3',x_{11}'$ we denote the points satisfying $f(\underline{x_2})=f(x_2')$ etc.}
	\label{pic:example}
\end{sidewaysfigure}

\begin{sidewaysfigure}[h]
	\vspace{9cm}
	\captionsetup{singlelinecheck=off}
	\includegraphics[width=\textwidth]{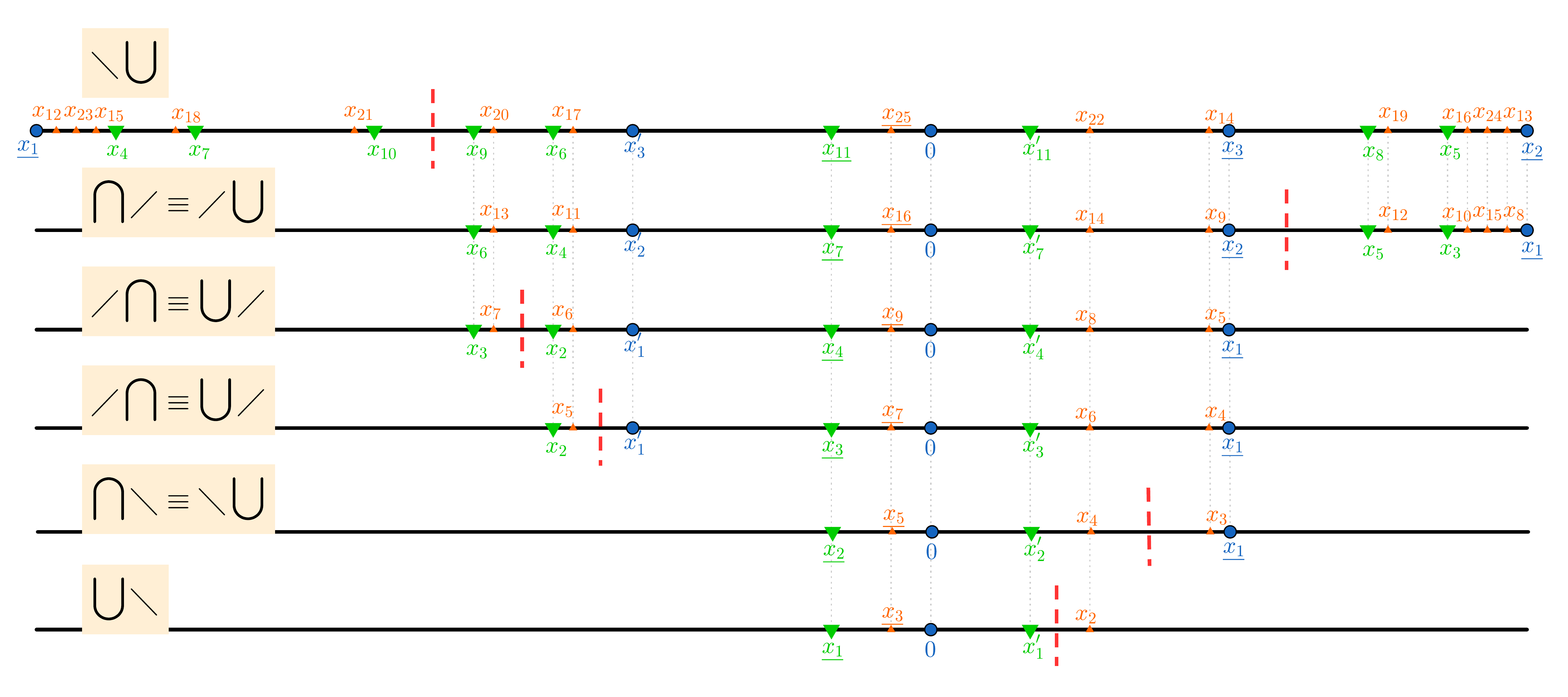}
	\caption[foo bar]{First 6 renormalizations for $\theta=[1,1,1,3,2,...]$ and type $\diagdown\bigcup$. For convenience scalings are avoided (that is why the non-normalized types with $\bigcap$ are present). Red dotted line depicts separation between $J$ and $T$, i.e.\ between the unimodal and the monotonic parts.
		
	Description by steps:
	\begin{enumerate}
		\item $\theta=[1,1,1,3,2,...]$, best recurrence times are $q_1=1,q_2=2,q_3=3,q_4=11,q_5=25,...$;
		\item $[1,1,3,2,...]$, $q_1=1,q_2=2,q_3=7,q_4=16,...$;
		\item $[1,3,2,...]$, $q_1=1,q_2=4,q_3=9,...$;
		\item $[3,2,...]$, $\mathbf{q_0=1},q_1=3,q_2=7,...$;
		\item $[2,2,...]$, $\mathbf{q_0=1},q_1=2,q_2=5,...$;
		\item $[1,2,...]$, $\mathbf{q_0=1},q_1=3,...$.
	\end{enumerate}
	
	}
	\label{pic:renorm}
\end{sidewaysfigure}


\begin{thebibliography}{RSSS}
	
	\bibitem[A]{Arnoux} Pierre Arnoux, \emph{Chap. 6: Sturmian sequences}. In book:
	Substitutions in Dynamics, Arithmetics and Combinatorics, pp. 143-198.
	
	\bibitem[LM]{LM} Mikhail Lyubich and John Milnor, \emph{The Fibonacci unimodal map}. Journal of the American Mathematical Society, Vol. 6 , No. 2 (1993), pp. 425-457.
	
	\bibitem[dMvS]{deMelo_vanStrien} Welington de Melo and Sebastian van Strien, \emph{One-Dimensional Dynamics}. Springer (1993).
	
	\bibitem[MT]{MT} John Milnor and William Thurston, \emph{On iterated maps of the interval}. In book: Dynamical Systems, pp.465-563.
	
	\bibitem[S]{S} Dennis Sullivan, \emph{Bounds, quadratic differentials, and renormalization conjectures}. Proceedings of the AMS Centennial Symposium (1988).
	
\end{thebibliography}
\end{document}